\documentclass[twoside,11pt, a4paper]{article}
\usepackage{enumitem}
\setlist{noitemsep, topsep= 0pt, parsep=1pt, partopsep=1pt, leftmargin=0.1cm}
\normalfont
\usepackage{multirow}
\usepackage{natbib}
\usepackage{amsthm}

\usepackage[pagebackref]{hyperref}

\newtheorem{theorem}{Theorem}[section]
\newtheorem{lemma}[theorem]{Lemma} 
\newtheorem{remark}[theorem]{Remark}
\newtheorem{definition}[theorem]{Definition}

\usepackage[a4paper, total={175mm, 295mm}, bottom= 40mm,  top= 30mm]{geometry}

\usepackage[thinlines]{easytable}
\usepackage{hyperref}
\hypersetup{colorlinks=true, linkcolor=blue, urlcolor={blue},
   pdfstartview={FitV},unicode,breaklinks=true, citecolor = blue}
\usepackage{amsmath}
\usepackage[font=scriptsize]{caption}
\usepackage{array}

\usepackage[utf8]{inputenc}
\usepackage{geometry}
\usepackage{amssymb}
\usepackage{breqn}

\title{\textbf{Testing for exponentiality for stationary  associated random variables}}
\vspace{-.8in}
\author{\small$\text{ Mansi Garg}\footnote{Corresponding author. }$ ${ }$ and $\text{Isha Dewan}.$ \vspace{-0.1in}\\
${}$\small Indian Statistical Institute, \vspace{-0.1in} \small Bangalore-560059 (India) and  \\
 ${}$\small Indian Statistical Institute, \vspace{-0.1in} \small New Delhi-110016 (India).\\ \small mansibirla@gmail.com and isha@isid.ac.in $\vspace{-0.8in}$ 
 }
\begin{document}
\fontsize{7}{10}
\date{}
\maketitle
\abstract{In this paper, we consider the problem of testing for exponentiality against univariate positive ageing when the underlying sample consists of stationary associated random variables.   In particular, we discuss the asymptotic behavior of the tests  by \cite{DESHPANDE01081983}, \cite{hollander1972testing} and \cite{ahmad1992} for testing exponentiality against IFRA, NBU and DMRL, respectively under association.  A simulation study illustrates the effect of dependence on the asymptotic normality of the test statistics and on the size and power of the tests. }

\numberwithin{equation}{section}
\section{Introduction}

The need to test for exponentiality against various univariate ageing classes occurs in many fields of research, such as reliability and survival analysis, queueing theory and economics among others. Traditionally, for testing for exponentiality  it was assumed that the the random variables of interest are independent and identically distributed $(i.i.d.)$.  However, in many real applications the assumption of independence is seldom satisfied.  The aim of this paper is to discuss the testing problem when the underlying random variables are associated. 

In the following, we discuss   the popular ageing classes studied in the paper, the concept of association  and various examples of associated random variables   occurring in the literature, and  then finally the tests for exponentiality against the  ageing classes under association. 

In reliability analysis, interest often lies in studying the ageing concepts of the lifetime of a component or a system as these help to analyze how it improves or deteriorates with time. Let $X$ be the lifetime of the component/ system under consideration with the distribution function $F(x)$ ($F(x) = 0$, $x < 0$), the survival function $\bar{F}(x)$, and the probability density function $f(x), x \ge 0$.  The failure rate function and the  mean residual lifetime function associated with $X$ are defined as $r(x) = f(x)/\bar{F}(x)$ and $\mu(x) = \int_{0}^{\infty} \frac{\bar{F}(x +t)}{\bar{F}(x)}dt$, whenever $\bar{F}(x) > 0$, $x \ge 0$, respectively. The ageing concepts are often described via the characteristics of the functions $\bar{F}(x)$, $r(x)$, and $\mu(x)$. 

Depending on the behavior of the chosen ageing criteria, the lifetime distribution can be categorized into various ageing classes. ``No ageing''  is synonymous to the lifetime distribution being exponential.  Positive (negative) ageing occurs when the system or component under consideration deteriorates (improves) over time. Some of the widely used classes of  positive ageing include the class of ``Increasing Failure Rate Average (IFRA)'', the class of ``New Better than Used (NBU)'', and the class of ``Decreasing Mean Residual Lifetime (DMRL)''. The negative dual of these classes are  DFRA, NWU, and IMRL respectively. These classes are defined as follows.

\begin{definition} $F$ is said to be IFRA (DFRA)  if $-(1/x)log\bar{F}(x)$  is increasing (decreasing)  in $x \ge 0$. This is equivalent to $\bar{F}(bx) \ge \bar{F}^b(x)$ $(\bar{F}(bx) \le \bar{F}^b(x))$ , $0 < b < 1$, $x \ge 0$.
\end{definition}
\begin{definition} $F$ is said to be NBU (NWU) if $\bar{F}(x+t) \le \bar{F}(x)\bar{F}(t)$ $(\bar{F}(x+t) \ge \bar{F}(x)\bar{F}(t))$ for all $x, t \ge 0$ and strict inequality for some $x, t \ge 0$.
\end{definition} 
\begin{definition} $F$ is said to be DMRL (IMRL) if the Mean Residual Life (MRL) function  $\mu(x)$ is decreasing (increasing)  in $x$, i.e., $\mu(s) \ge \mu(t)$ $(\mu(s) \le \mu(t))$  for $0 \le s \le t$.
\end{definition}    

 Optimal maintenance, replacement, and resource allocation  policies can be separately designed for each family of distributions. The knowledge of the lifetime belonging to a particular class of distributions can be used to choose appropriate parametric or a constrained nonparametric model for the underlying ageing process.

Testing for exponentiality against different ageing alternatives is also useful in  queueing theory.  For example, the service times and inter-arrival times in the classical queueing model, $M/M/1$, are assumed to come from mutually independent sequences of $i.i.d.$ exponential random variables. It leads to analytically tractable  expressions of the performance metrics, like the mean number of customers in the system, and the mean service and arrival rates. Extensions of the classical model include $M/G/1$, $M/G/k$,   and $M/G/\infty$, $G/G/k$ and $G_t/G/k$ among others. In all the models, the service times  have a \textbf{G}eneral distribution.  Several queueing models also assume that the inter-arrival times  have a general distribution.  For example, the queueing model $G/G/1$. In most queueing models, the probability distribution of the service times and the inter-arrival  times impact the output characteristics. Hence, the knowledge of the service time and the inter-arrival times belonging to a particular class of distributions  is useful in  developing a queueing model for the underlying system to determine its long term behavior.  For example, in \cite{abramov2006stochastic} stochastic inequalities for the number of losses for some single-server queueing models when the inter-arrival times or the services times are NBU or NWU have been derived. 

The classification of  distributions into various ageing classes is also of interest to researchers  in economics.  An application  is in testing for the duration dependence (see \cite{Ecojournal}). Another possible application is in choosing the appropriate marginal distribution for modeling various time series data. For example, processes like $GARCH$ and $ARCH$ with heavy tailed marginal distributions have been used to model many financial time series.

Many tests exist in literature that test for exponentiality (or the assumption of constant failure rate) against different positive or negative  ageing alternatives.  A detailed discussion on the various classes of ageing along with their testing procedures and applications for  $i.i.d.$ random variables  can be found in \cite{deshpande2005life} and \cite{lai2006stochastic}.   However, in many real applications,  the random variables under consideration are dependent.  

For example,  in reliability analysis, the  lifetimes of  independent components in a reliability structure when the components  share the same load or  are subject to a shared environmental stress are dependent (see  \cite{barlow1975statistical} and \cite{5976637}). Various autoregressive models with minification structures have positively correlated components.  For example, let $X_0$ be a non-degenerate and  non-negative random variable, and  $\{\epsilon_n, n \in \mathbb{N}\}$ be a sequence of independent and identically distributed $(i.i.d.)$, non-negative and non-degenerate   random variables  independent of $X_0$. Then, the non-negative random variables
\begin{equation*}
  X_n= k \: min(X_{n-1}, \epsilon_n) \text{ for all }n \in \mathbb{N}  \text{ and for some $k > 1$},\end{equation*}
 are  dependent. Minification processes have been used to model dependent lifetime data (for example, see \cite{CORDEIRO2014465}) and dependent service times (for example, see \cite{mironimpact}). 

In all these cases, the random variables  under consideration are associated -  a concept defined by   \cite{MR0217826} as follows. 

   \begin{definition} A finite collection of random variables $\{X_j,  1 \le j \le n \}$  is said to be associated, if for any choice of component-wise non-decreasing functions $h$, $g$ $:$ ${\mathbb{R}}^n \rightarrow {\mathbb{R}}$, we have,
\begin{equation*}
Cov(h(X_1,\ldots, X_n), g(X_1,\ldots, X_n)) \ge 0
\end{equation*}
whenever it exists. An infinite collection of random variables $\{X_j,  j \ge 1 \}$ is associated if every finite sub-collection is associated. 
\end{definition}
Any set of independent random variables is associated (\cite{MR0217826}). Non-decreasing functions of associated random variables are associated, for example, order statistics corresponding to a finite set of independent random variables are associated (\cite{MR0217826}).  Few other examples of associated random variables are:  positively correlated normal random variables (\cite{pitt1982}); the components of \cite{marshall12345}  multivariate exponential distribution, multivariate extreme-value distribution (\cite{marshall1983}) and  Downton multivariate exponential distribution (\cite{downton}); the components of the moving average process $\{X_n= a_0\epsilon_n+  a_1\epsilon_{n-1}, n \in \mathbb{N}\}$, where $\epsilon_n$, $n \in \mathbb{N}\cup \{0\}$ are independent random variables and $a_0$, $a_1$ have the same sign.  A detailed compilation of results and applications for associated random variables can be found in \cite{Bulinski:1701595}, \cite{MR3025761} and \cite{oliveira2012asymptotics}.

While the control of dependence in stochastic processes is generally  given in terms of mixing conditions, an obvious drawback is that the mixing coefficients are defined using $\sigma$-fields.  It makes these coefficients  difficult to  compute in practice.   For  associated random variables, the  control of dependence is through the covariance structure of the  random variables.  The simplicity of the conditions under which  the limit theorems can be proved gives an advantage  over the popularly used mixing processes. 

In this paper, we discuss the limiting behavior of some of the tests of exponentiality against univariate positive ageing based on U-statistics when the underlying random variables are stationary and associated.   In particular, we look at tests by \cite{DESHPANDE01081983}, \cite{hollander1972testing} and \cite{ahmad1992} for testing exponentiality against IFRA, NBU and DMRL  respectively.  The kernels of the test statistics of the given tests  belong to the class of kernels which are bounded (but  are not  of bounded variation). For tests based on U-statistics for $i.i.d.$ random variables, the test statistics can be shown to be asymptotically normally distributed using the results of  \cite{MR0026294}. However, it is not possible to directly extend  the theory of asymptotic normality for  U-statistics based on dependent random variables.  Hence, the asymptotic behavior of U-statistics for associated random variables needs to be looked into separately.  

We first develop a central limit  theorem  for U-statistics based on the class of kernels discussed above for  stationary associated random variables.  We next use this result to obtain critical points, size and power for  the given tests. This helps in analyzing the behavior of the considered  tests under the dependent setup. 

For the rest of the paper, assume $\{X_n, n \in \mathbb{N}\}$ is a stationary sequence of associated random variables with  the distribution function of $X_1$ denoted by $F$.  We also assume that  $X_n, n \in \mathbb{N}$ are uniformly bounded, $i.e.$ there exists a $ 0 < C_1 < \infty$, such that $P(|X_1| \le C_1) = 1$. For applications in reliability and survival analysis this assumption is reasonable. 

 The paper is organized as follows. In the next section, Section $\ref{CLT}$,  we give a general theorem for the asymptotic distribution of U-statistics based on bounded kernels for stationary associated random variables.  In Section $\ref{tests}$, we apply this result to discuss the limiting behavior of the tests by \cite{DESHPANDE01081983}, \cite{hollander1972testing} and \cite{ahmad1992} under association.  In Section $\ref{sim}$, the asymptotic normality of the test statistics and the power of the tests under the discussed dependent setup is illustrated via simulations.  Section $\ref{dis}$ is a brief discussion on the applications of our results and our intended future work. Section $\ref{proofs}$ contains some preliminary results and the  proofs of our technical results. 

\section{Central limit theorem for U-statistics based on bounded kernels}\label{CLT}

The main result of this section, Theorem $\ref{theorem25}$,  gives the central limit theorem for U-statistics based on a bounded kernel of degree $2$, when the underlying sample is  a sequence of stationary associated random variables.  The extension of this theorem to U-statistics with kernels of a general finite degree k $>$ $2$ are also discussed. These results are applied in section $\ref{tests}$ to show the asymptotic normality of the test statistics  of the considered tests of exponentiality against positive ageing under the dependent setup.   Proof of the results are postponed to section $\ref{proofs}$.

 The central limit theorem  for  U-statistics discussed extends the results of \cite{MR1861150, Dewan20029, Dewan2015147}, \cite{Garg2015209, GargdewanSPL2} to a wider class of kernels.

The U-statistic $U_n(\rho)$ of degree $2$ based on $\lbrace{X_j, 1 \le j \le n}\rbrace$  $( n \ge 2)$ with a symmetric kernel $\rho: \mathbb{R}^2 \to \mathbb{R}$ is defined as,
\begin{equation}
U_n(\rho) = {{n}\choose{2}}^{-1}\sum_{1\le {j_1}< {j_2}\le n} {\rho}(X_{j_1}, X_{j_2}).
\end{equation}

 Let $\theta = \int\limits_{\mathbb{R}^2}{\rho}(x_1, x_2)~{dF({x_1})dF({x_2})}$.  Define
\begin{align*}
&  {\rho}_{1}({x_1}) = \int\limits_{\mathbb{R}}{\rho}(x_1, x_2)~dF({x_{2}}), \: \:  h^{(1)}({x_1}) =  {\rho}_1({x_1}) - \theta,\\
&  \text{and }   \: h^{(2)}({x_1}, {x_2}) =  {\rho}({x_1}, {x_2}) - {\rho}_1({x_1}) - {\rho}_1({x_2}) + \theta.
  \end{align*}
   Then, the Hoeffding-decomposition for $U_n(\rho)$ is $U_n(\rho) = \theta + 2H_n^{(1)} + H_n^{(2)}$,
    where $H_{n}^{(j)}$ is the U-statistic of degree $j$ based on the kernel $h^{(j)}$, $j = 1, 2$. When the observations are $i.i.d.$, $E(U_n(\rho)) = \theta$.

Similarly, the Hoeffding's decomposition for U-statistics of a finite  degree $k > 2$ can be obtained.

\subsection{Central limit theorem}

Before proceeding, we need to define the following.

Let $f_{\textbf{Z}}$ denote the $p.d.f.$, and let $\Phi_{\textbf{Z}}$ denote the characteristic function of  random vector  $\textbf{Z} \in \mathbb{R}^k$, respectively. 

\begin{definition}The $p.d.f$ $f_{\textbf{Z}}$ of the random vector  $\textbf{Z}$ is said to satisfy a Lipshitz condition of order 1, if for every $\textbf{x}, \textbf{u} \in \mathbb{R}^k$ and some finite constant $C > 0$,
\begin{align}\label{lip}
|f_{\textbf{Z}}(\textbf{x + u}) - f_{\textbf{Z}}(\textbf{x })| \le C \sum_{j=1}^k |u_j|.
\end{align}
\end{definition}

\begin{definition} (\cite{MR789244})  If $f$ and $\tilde{f}$ are two real-valued functions on $\mathbb{R}^{n}$, then $ f \ll \tilde{f}$ iff $\tilde{f} + f$ and $\tilde{f} - f$ are both coordinate-wise non-decreasing. 
\end{definition}

If $f \ll \tilde{f}$, then $\tilde{f}$ will be coordinate-wise non-decreasing. 

We next define the conditions $(T1)$ and $(T2)$ that will be needed to prove Theorem $\ref{theorem25}$. 

 In the following, let $\lbrace{X_n', n \in \mathbb{N}}\rbrace$ be a  sequence of $i.i.d.$ random variables independent of the sequence $\lbrace{X_n, n \in \mathbb{N}}\rbrace$,  with the marginal distribution function of $X_1'$ being $F$. 

   \begin{itemize}
   \item[(T1)] For all distinct $i_1, i_2, i_3, i_4$, such that $1 \le i_1 < i_2 \le n$ and $1 \le i_3 < i_4 \le n$, 
    \begin{itemize}
 \item[(I1)] $f_{X_{i_1},  X_{i_2}, X_{i_3}, X_{i_4}}$, $f_{X_{i_1}', X_{i_2}, X_{i_3}, X_{i_4}}$, $f_{X_{i_1}, X_{i_2}', X_{i_3}, X_{i_4}}$ and $f_{X_{i_1}, X_{i_2}, X_{i_3}', X_{i_4}}$  are bounded and satisfy the  Lipshitz condition of order 1 and
  \item[(I2)] $\Phi_ {X_{i_1},  X_{i_2}, X_{i_3}, X_{i_4}}$, $\Phi_ {X_{i_1}',  X_{i_2}, X_{i_3}, X_{i_4}}$, $\Phi_ {X_{i_1},  X_{i_2}', X_{i_3}, X_{i_4}}$ and $\Phi_ {X_{i_1},  X_{i_2}, X_{i_3}', X_{i_4}}$  are absolutely integrable. 
  \end{itemize}
  \item[(T2)]     For any 3 distinct indices $i, j, k$ from $(i_1, i_2, i_3, i_4)$, such that $1 \le i_1 < i_2 \le n$ and $1 \le i_3 < i_4 \le n$, 
    \begin{itemize}
 \item[(J1)] $f_{X_{i},  X_{j}, X_{k}}$ and $f_{X_{i}',  X_{j}, X_{k}}$,   are bounded and satisfy the  Lipshitz condition of order 1 and, 
  \item[(J2)] $\Phi_ {X_{i},  X_{j}, X_{k}}$ and $\Phi_ {X_{i}',  X_{j}, X_{k}}$  are absolutely integrable. 
  \end{itemize}
    \end{itemize}

\begin{theorem} \label{theorem25} Let   $U_n(\rho)$ be the U-statistic based on a symmetric kernel $\rho(., .)$ which is   bounded ($i.e.$ $|\rho(x, y)|\le C_2$, for some $C_2 <\infty$ for all $x, y \in \mathbb{R}$). Define $\theta = \int \int \rho(x, y)dF(x)dF(y)$, \\$\sigma_1^2 = Var(\rho_1(X_1))$ and $\sigma_{1j} = Cov( \rho_1(X_1), \rho_1(X_{1+j}))$ for all $j \in \mathbb{N}$.  

Assume the following.
\begin{itemize}
\item[(i)] $\sum_{j=1}^\infty Cov({X}_1, {X}_j)^{\frac{1}{21}} < \infty$;
\item[(ii)]  $(T1)$ and $(T2)$ hold; and
\item[(iii)] $\sigma_1^2 < \infty$ and  $\sum_{j=1}^{\infty}|\sigma_{1j}| < \infty$.
\end{itemize}
Then
\begin{equation}\label{refnormvar}
Var(U_n(\rho)) = \frac{4\sigma^2_U}{n} + o\Big(\frac{1}{n}\Big), \text{ where } \sigma^2_U = \sigma^2_1 + 2\sum_{j=1}^{\infty}\sigma^2_{1j}.
\end{equation}
Further, if  $\sigma^2_U > 0$ and  there exists a function $\tilde{\rho}_1: \mathbb{R} \to \mathbb{R}$, such that ${\rho}_1$ $\ll$ $\tilde{\rho}_1$ and
\begin{equation}
\sum_{j=1}^{\infty}Cov(\tilde{\rho}_1({X}_1), \tilde{\rho}_1({X}_{j})) < \infty,
\end{equation}
 then
\begin{equation}\label{refnormpdf}
\frac{\sqrt{n}( U_n(\rho) - \theta)}{2\sigma_U} \xrightarrow {\mathcal{ L}} N(0, 1) \:\: \text{as}\: \:{n \to \infty}.
\end{equation}
\end{theorem}

\begin{remark} \label{remark22} Theorem $\ref{theorem25}$ can be easily extended to a U-statistic based on a kernel of any finite degree $k > 2$. Let $U_n(\rho)$ be the U-statistic based on the symmetric kernel $\rho(x_1, x_2,\dots, x_k)$ which is bounded. 

Let $\lbrace{X_n', n \in \mathbb{N}}\rbrace$ be a  sequence of $i.i.d.$ random variables independent of the sequence $\lbrace{X_n, n \in \mathbb{N}}\rbrace$,  with the marginal distribution function of $X_1'$ being $F$. 

Assume for all $p = 2, \cdots, k$ the following are true.

\begin{itemize}
\item For all  distinct indices $i_1, i_2, \cdots, i_{2p}$, such that $1 \le i_1 < \cdots < i_{p} \le n$ and $1 \le i_{p+1} < \cdots < i_{2p} \le n$, 
    \begin{itemize}
 \item[(1)] $f_{X_{i_1},  \cdots X_{i_{2p}}}$, $f_{X_{i_1}', X_{i_2},\cdots, X_{i_{2p}}}$, $f_{X_{i_1}, X_{i_2}', X_{i_3},\cdots, X_{i_{2p}}}$ and  $f_{X_{i_1}, X_{i_2}, X_{i_3}', X_{i_4}, \cdots, X_{i_{2p}}}$ are bounded and satisfy the  Lipshitz condition of order 1, and 
  \item[(2)] $\Phi_{X_{i_1},  \cdots X_{i_{2p}}}$, $\Phi_{X_{i_1}', X_{i_2},\cdots, X_{i_{2p}}}$, $\Phi_{X_{i_1}, X_{i_2}', X_{i_3},\cdots, X_{i_{2p}}}$  and $\Phi_{X_{i_1}, X_{i_2}, X_{i_3}', X_{i_4}, \cdots, X_{i_{2p}}}$ are absolutely integrable. 
  \end{itemize}
    \item For $(2p-1)$ distinct indices $j_1, j_2, \cdots, j_{2p-1}$ from  $(i_1, i_2, \cdots, i_{2p})$, such that $1 \le i_1 < \cdots < i_{p} \le n$ and $1 \le i_{p+1} < \cdots < i_{2p} \le n$, 
     \begin{itemize}
 \item[(1)] $f_{X_{j_1},  \cdots X_{j_{2p-1}}}$ and $f_{X_{j_1}', X_{i_2},\cdots, X_{j_{2p-1}}}$ are bounded and satisfy the  Lipshitz condition of order 1, and
  \item[(2)] $\Phi_{X_{j_1},  \cdots X_{j_{2p-1}}}$ and $\Phi_{X_{j_1}', X_{j_2},\cdots, X_{j_{2p-1}}}$ are absolutely integrable. 
  \end{itemize}
    \end{itemize}
 Further, if $\sum_{j=1}^\infty Cov({X}_1, {X}_j)^{\frac{1}{3(3+2k)}}$ $<$ $\infty$,  $\sigma_1^2$ $<$ $\infty$ and $\sum_{j=1}^{\infty}|\sigma_{1j}| <$ $\infty$, then
\begin{equation}
Var(U_n(\rho)) = \frac{k^{2}\sigma^2_U}{n} + o\Big(\frac{1}{n}\Big).
\end{equation}
If $\sigma^2_U > 0$ and there exists a function $\tilde{\rho}_1: \mathbb{R} \to \mathbb{R}$, such that ${\rho}_1$ $\ll$ $\tilde{\rho}_1$ and
\begin{equation}
\sum_{j=1}^{\infty}Cov(\tilde{\rho}_1({X}_1), \tilde{\rho}_1({X}_{j})) < \infty,
\end{equation}
then
\begin{equation}
\frac{\sqrt{n}( U_n(\rho) - \theta)}{k\sigma_U} \xrightarrow {\mathcal{ L}} N(0, 1) \:\: \text{as}\: \:{n \to \infty}.
\end{equation}  
 \end{remark}
 \begin{remark} The results can be extended to non-uniformly bounded random variables under stricter covariance restrictions, by using the standard truncation technique and putting appropriate assumptions on the moments of the underlying random variables.
 \end{remark}

\section{Tests for ageing}\label{tests}
\setlength\parindent{24pt}
\cite{DESHPANDE01081983}, \cite{hollander1972testing} and \cite{ahmad1992} had proposed tests for testing exponentiality against IFRA, NBU and DMRL respectively, for a sample of  $i.i.d.$ observations.  In this section, we prove the asymptotic normalty of the test statistics of  these tests when the underlying sample consists of stationary associated random variables.  

Let $\mu = \mu(0) = E(X_1)$ in the following.

  \subsection{{Testing Exponentiality against  IFRA alternatives}}
Our aim is to test 
\begin{equation*}
H_0: F(x) = 1-exp(-x/\mu),  \hspace{3 mm} x\ge \: 0, \text{ }\mu > 0, \text{ against}
\end{equation*}
\begin{equation*}
 H_1: \text{F is IFRA but not exponential}.
\end{equation*}
The test statistic, $J_{(n, b)}$ $(0 < b < 1)$, of the test proposed by \cite{DESHPANDE01081983} is
\begin{equation}
J_{(n, b)} = \frac{1}{{n\choose2} }\sum_{1 \le i < j \le n}\rho(X_i, X_j),
\end{equation}
where
\begin{equation*}
\rho(x, y) = \frac{ h_1(x, y) + h_1(y, x)}{2},
\end{equation*}
and
\[h_1(x,y)  = \left \{\begin{array}{ll}
1 & \mbox {if}\; x > by \nonumber \\
 0 & \mbox{otherwise}.
 \end{array} \right. \]

When $\{X_j, 1\le j \le n\}$ are $i.i.d.$, the asymptotic distribution of $J_{(n, b)}$ under $H_0$, as discussed in \cite{DESHPANDE01081983} is\\
\begin{equation}
\frac{\sqrt{n}(J_{(n, b)} - \frac{1}{b+1}) }{2\sqrt{{\xi}_1}} \xrightarrow{\mathcal{ L}}  N(0,1), \text{ as $n \to \infty$},
\end{equation}
where
\begin{equation*}
\xi_1 = Var(\rho_1(X_1))= \frac{1}{4}\Bigg{\{} 1+ \frac{b}{2+b} +\frac{1}{2b+1} + \frac{2(1-b)}{(1+b)} - \frac{2b}{(1+b+b^2)}-\frac{4}{(b+1)^2} \Bigg{\}}, 
\end{equation*}
and
\begin{equation}
\rho_1(x) = \frac{\bar{F}(bx) + F(\frac{x}{b})}{2}, \: x \ge 0.
\end{equation}
We now obtain a limiting distribution for $J_{(n, b)}$ when the observations are associated. 
\begin{theorem}
Let $\{X_n, n \ge 1\}$ be a sequence of stationary associated random variables, such that $P(|X_n| \le C_1) = 1$, for some $0 < C_1 < \infty$, for all $n \ge 1$. Assume that conditions of Theorem $\ref{theorem25}$ are satisfied. Then, the limiting distribution of $J_{(n, b)}$ under $H_0$ is
\begin{equation*}
D_{(n, b)} = \frac{\sqrt{n}( J_{(n, b)} - \frac{1}{b+1} )}{2\sigma_D}  \xrightarrow {\mathcal{ L}} N(0, 1) \:\: \text{as}\: \:{n \to \infty},
\end{equation*}
 where  $\sigma_D^2 = Var(\rho_1(X_1)) + 2\sum_{j=1}^{\infty}Cov(\rho_1(X_1), \rho_1(X_{j+1}))$.
\proof

By Hoeffding's decomposition, $J_{(n, b)} = \theta + 2H^{(1)}_n + H^{(2)}_n$.

The underlying kernel $\rho$ is  not continuous and not of local bounded variation. However, it is bounded. Also, $\rho_1$ is a Lipschitz function (i.e ${\rho}_1(x)$ $\ll$ $Cx$, for all $x \in [-C_1, C_1]$ and for some $C > 0$). From Theorem $\ref{theorem25}$,
\begin{equation*}
D_{(n, b)}  \xrightarrow {\mathcal{ L}} N(0, 1) \:\: \text{as}\: \:{n \to \infty}.
\end{equation*}
\end{theorem}

\emph{Rejection criteria}: Since ${\sigma}_D$ is unknown, we use the following test statistic
\begin{equation*}
\hat{D}_{(n, b)} = \frac{\sqrt{n}( J_{(n, b)} - \frac{1}{b+1} )}{2{\hat{\sigma}}_D},
\end{equation*}
where $\hat{\sigma}_D$ is a consistent estimator for $\sigma_D$. Reject $H_0$ at a significance level $\alpha$ if $\hat{D}_{(n, b)}$ $\ge$ $z_{1-\alpha}$, where $z_{1- \alpha}$ is  $100(1-\alpha)^{th}$ percentile of $N(0,1)$.

 \subsection{{Testing exponentiality against NBU alternatives}}
Our aim is to test 
\begin{equation*}
H_0: \bar{F}(s+t) = \bar{F}(s)\bar{F}(t), \: \: \: s, t \ge 0, \text{ (i.e $F$ is exponential)} \text{ against}
\end{equation*}
\begin{equation*}
 H_2:\bar{F}(s+t) \le \bar{F}(s)\bar{F}(t), \: \: \: s, t \ge 0; \text{ with strict inequality for some $s, t$.}
 \end{equation*}
The test by \cite{hollander1972testing} rejects $H_0$ for small values of the statistic, $S_n$, defined by $S_n = \frac{n(n-1)(n-2)}{2}N_n$, where
\begin{equation}
N_n =  \frac{6}{n(n-1)(n-2)} \sum_{1 \le i <j <k \le n} \rho (X_i, X_j, X_k).
\end{equation}
$\rho (x_1, x_2, x_3) = \frac{1}{3} [ \phi(x_1, x_2, x_3) + \phi(x_2, x_1, x_3) + \phi(x_3, x_1, x_2)]$ and 
 \[ \phi(x_1, x_2, x_3) = \left\{ 
  \begin{array}{l l}
    1 & \quad \text{if $x_1 > x_2 + x_3,$}\\
    0 & \quad \text{otherwise.}  \end{array} \right. \]

When $\{X_j, 1\le j \le n\}$ are $i.i.d.$, the asymptotic distribution of $N_n$ can be obtained by the central limit theorem for U-statistics as discussed in \cite{MR0026294}. In particular, under $H_0$, we get
\begin{equation}
\frac{\sqrt{n}( N_n - \frac{1}{4} )}{\sqrt{5/432}} \xrightarrow {\mathcal{ L}} N(0, 1) \:\: \text{as}\: \:{n \to \infty}.
\end{equation}
 The kernel is of degree $3$. 
 \begin{equation}
 \rho_1(x) = \frac{1}{3}\Big( \int_0^{x}F(x-z)dF(z) + \int_0^{\infty}\bar{F}(x+z)dF(z) + \int_{x}^{\infty}F(z-x)dF(z)\Big)
  \end{equation}
 for $x \ge 0$.
 \begin{theorem}
Let $\{X_n, n \ge 1\}$ be a sequence of stationary associated random variables, such that  $P(|X_n| \le C_1) = 1$, for some $0 < C_1 < \infty$, for all $n \ge 1$. Under the conditions discussed in Remark $\ref{remark22}$ (k =3) and $H_0$
\begin{equation*}
HP_n =\frac{\sqrt{n}( N_n - \frac{1}{4} )}{3{{\sigma}}_{HP}} \xrightarrow {\mathcal{ L}} N(0, 1) \:\: \text{as}\: \:{n \to \infty},
\end{equation*}
where $\sigma^2_{HP} = Var(\rho_1(X_1)) + 2\sum_{j=1}^{\infty}Cov(\rho_1(X_1), \rho_1(X_{j+1}))$.
  \end{theorem}

  \emph{Rejection criteria}: Since ${\sigma}_{HP}$ is unknown, we use the following test statistic
\begin{equation*}
\hat{HP}_n = \frac{\sqrt{n}( N_n - \frac{1}{4} )}{3{\hat{\sigma}}_{HP}},
\end{equation*}
where $\hat{\sigma}_{HP}$ is a consistent estimator for $\sigma_{HP}$. Reject $H_0$ at a significance level $\alpha$ if  $\hat{HP}_n$ $\le$ $z_{\alpha}$, $z_{\alpha} = -z_{(1 - \alpha)}$.

 \subsection{{Testing exponentiality against DMRL alternatives}}

 Assume the MRL function $\mu(x)$ is differential and $x\bar{F}^2(x) \to 0$ as $x \to \infty$. Our aim is to test 
\begin{equation*}
H_0: \mu(x) \text{ is constant}  \text{ (i.e $F$ is exponential), against}
\end{equation*}
\begin{equation*}
 H_3: \frac{d\mu{(x)}}{dx} \le 0 \text{ or } f(x)\int_x^{\infty}\bar{F}(u)du \le \bar{F}^2(x), x \ge 0.
\end{equation*}
A test by \cite{ahmad1992} rejects  $H_0$ in favor of $H_3$ for large values of $\delta_{F_n}$, where
\begin{equation}
\delta_{F_n} =  \frac{1}{{n \choose 2}} \sum_{1 \le i_1 < i_2  \le n} \rho(X_{i_1}, X_{i_2}),
\end{equation}
and $\rho(x_1, x_2) = \frac{1}{2}[\phi(x_1, x_2) + \phi(x_2, x_1)]$. Here,
 \[ \phi(x_1, x_2) = \left\{ 
  \begin{array}{l l}
    (3x_1 - x_2) & \quad \text{if $x_2> x_1,$}\\
    0 & \quad \text{otherwise.}  \end{array} \right. \]

When $\{X_j, 1\le j \le n\}$ are $i.i.d.$, the asymptotic distribution of $\delta_{F_n}$  can be obtained by the central limit theorem for U-statistics as discussed in \cite{MR0026294}. In particular, under $H_0$, we get,
\begin{equation}\label{cltahmad}
\frac{\sqrt{n} {\delta_{F_n}}}{\mu \sqrt{1/3}} \xrightarrow {\mathcal{ L}} N(0, 1) \:\: \text{as}\: \:{n \to \infty}. 
\end{equation} 
The kernel is of degree $2$. 
\begin{align}
\rho_1(x) & = \frac{1}{2} \Big( \int_x^{\infty}(3x-y)dF(y) +  \int_0^{x}(3y-x)dF(y) \Big)  = 2x\bar{F}(x) - \frac{x}{2} + \frac{3\mu}{2} - 2\int_x^{\infty}ydF(y).
  \end{align}
The statistic $\delta_{F_n}$ can be made scale invariant by considering ${\delta_{F_n}}/{\bar{X}_n}$ ($\bar{X}_n = \sum_{j=1}^n \frac{X_j}{n}$). The limiting distribution of ${\delta_{F_n}}/{\bar{X}_n}$ under $H_0$ follows using $(\ref{cltahmad})$ and the Slutsky's theorem, i.e
\begin{equation}
\frac{\sqrt{n} {\delta_{F_n}}}{\bar{X}_n \sqrt{1/3}} \xrightarrow {\mathcal{ L}} N(0, 1) \:\: \text{as}\: \:{n \to \infty}. 
\end{equation} 
 \begin{theorem}
Let $\{X_n, n \ge 1\}$ be a sequence of stationary associated random variables, such that  $P(|X_n| \le C_1) = 1$, for some $0 < C_1 < \infty$, for all $n \ge 1$. Assume that conditions of Theorem $\ref{theorem25}$ are satisfied. Then, under $H_0$,
\begin{equation*}
A_n = \frac{\sqrt{n}\delta_{F_n}}{2\sigma_A} \xrightarrow {\mathcal{ L}} N(0, 1) \:\: \text{as}\: \:{n \to \infty},
\end{equation*}
where $\sigma^2_A =  Var(\rho_1(X_1)) + 2\sum_{j=1}^{\infty}Cov(\rho_1(X_1), \rho_1(X_{j+1}))$.
  \end{theorem}
    
  \emph{Rejection criteria}: Since ${\sigma}_A$ is not known, we use the following test statistic
\begin{equation*}
\hat{A}_n = \frac{\sqrt{n}\delta_{F_n}}{2{\hat{\sigma}}_A},
\end{equation*}
where $\hat{\sigma}_A$ is a consistent estimator for $\sigma_A$. Reject $H_0$ at a significance level $\alpha$ if $\hat{A}_n$ $\ge$ $z_{1-\alpha}$.

\begin{remark} The kernels of the test statistics discussed are discontinuous and not of local bounded variation. The existing results on U-statistics by \cite{Garg2015209, GargdewanSPL2} cannot be used to obtain the limiting distribution of the statistics discussed under the dependent setup.
 
\end{remark}
\begin{remark} The test statistics  $\hat{D}_{(n, b)}$, $\hat{HP}_n$, and $\hat{A}_n$ (under appropriate rejection criteria) can also be used for testing exponentiality against  DFRA, NWU and IMRL respectively.
\end{remark}
  \section{Simulations}\label{sim}

We assessed the performance of  IFRA, NBU and DMRL tests based on $\hat{D}_{(n, b)}$, $\hat{HP}_n$, and $\hat{A}_n$ when the underlying observations are  stationary and associated via simulations. We generated associated random variables using the property that non-decreasing functions of independent random variables are associated. We used the statistical software R (\cite{manR}) for our simulations. 

\begin{itemize}
\item[(1)] We investigated the asymptotic normality of the statistics under $H_0$. The marginal distribution of $X_j$ was taken as ${F}(x) = 1 - e^{-x}, x \ge 0$, $j \ge 1$, i.e we take $\mu = 1$ (the 3 tests discussed do not depend on the choice of $\mu$). The samples $\lbrace{X_j, 1 \le j \le n}\rbrace$ were generated as follows.

\begin{itemize}
\item [(S1)] $(m = 2)$ $X_j = min(X_j, X_{j+1})$, where $\lbrace {X_j, j \ge 1} \rbrace$ were pseudo-random numbers from $Exp(1/2)$ generated using $rexp$ function in R.
 \item[(S2)]  $(m = 3)$  $X_j =min( X_j, X_{j+1}, X_{j+2})$, where $\lbrace {X_j, j \ge 1} \rbrace$ were pseudo-random numbers from $Exp(1/3)$ generated using $rexp$ function in R.
  \item[(S3)]  $(m = 5)$  $X_j =min( X_j, X_{j+1}, \cdots, X_{j+4})$, where $\lbrace {X_j, j \ge 1} \rbrace$ were pseudo-random numbers from $Exp(1/5)$ generated using $rexp$ function in R.
 \item [(S4)]  $(m = 10)$ $X_j = min(X_j, \cdots, X_{j+9})$, where $\lbrace {X_j, j \ge 1} \rbrace$ were pseudo-random numbers from  $Exp(1/10)$ generated using $rexp$ function in R.
 \end{itemize}
 \item[(2)] We also calculated the empirical power of the above tests for the following alternatives. 
\begin{itemize}
\item[(a)] The marginal distribution of $X_j$ was taken as ${F}_1(x) = 1- e^{-(e^{xa} -1)/a}, x \ge 0$, $ a > 0$, $j \ge 1$.  We took  $a = 0.5, 0.8, 1$. The samples $\lbrace{X_j, 1 \le j \le n}\rbrace$ were generated as follows.
\begin{itemize}
\item[(S5)] $(m = 2)$  $X_j = log( 1 + a\times min(X_j, X_{j+1}))/a$,  where $\lbrace {X_j, j \ge 1} \rbrace$ were pseudo-random numbers from $Exp(1/2)$ generated using $rexp$ function in R.
 \item[(S6)]  $(m = 3)$  $X_j =  log( 1 + a \times min( X_j, X_{j+1}, X_{j+2}))/a$, where $\lbrace {X_j, j \ge 1} \rbrace$ were pseudo-random numbers from $Exp(1/3)$ generated using $rexp$ function in R.
  \item[(S7)]  $(m = 5)$  $X_j =  log( 1 + a \times min( X_j, X_{j+1}, \cdots, X_{j+4}))/a$, where $\lbrace {X_j, j \ge 1} \rbrace$ were pseudo-random numbers from $Exp(1/5)$ generated using $rexp$ function in R.
 \item [(S8)]  $(m = 10)$  $X_j =  log( 1 +a \times min(X_j, \cdots, X_{j+9}))/a$, where $\lbrace {X_j, j \ge 1} \rbrace$ were pseudo-random numbers from  $Exp(1/10)$ generated using $rexp$ function in R.
 \end{itemize}

\item[(b)] The marginal distribution of $X_j$ was taken as ${F}_2(x) = 1- e^{-x -\frac{x^{2}}{a}}, x \ge 0$, $a > 0$, $j \ge 1$.    We took  $a = 10, 5, 2$. The samples $\lbrace{X_j, 1 \le j \le n}\rbrace$ were generated as follows.

\begin{itemize}
\item[(S9)] $(m = 2)$ $X_j = min(X_j, \sqrt{a_1X_{j+1}} )$,  where $\lbrace {X_j, j \ge 1} \rbrace$ were pseudo-random numbers from $Exp(1)$ generated using $rexp$ function in R.
  \item[(S10)]  $(m = 3)$  $X_j =  min(X_j, \sqrt{a_2X_{j+1}}, X_{j+2} )$, where $\lbrace {X_j, j \ge 1} \rbrace$ were pseudo-random numbers from $Exp(1/2)$ generated using $rexp$ function in R.
\item[(S11)]  $(m = 5)$  $X_j = min(X_j, \sqrt{a_3X_{j+1}}, X_{j+2}, X_{j+3}, \sqrt{a_3X_{j+4}})
$, where $\lbrace {X_j, j \ge 1} \rbrace$ were pseudo-random numbers from $Exp(1/3)$ generated using $rexp$ function in R.
\item[(S12)] $(m = 10)$  $X_j = min(X_{j}, \sqrt{a_4X_{j+1}}, X_{j+2}, X_{j+3}, \sqrt{a_4X_{j+4}}, X_{j+5}, X_{j+6}, \sqrt{a_4X_{j+7}},\\ \sqrt{a_4X_{j+8}}, \sqrt{a_4X_{j+9}}\: )$, where $\lbrace {X_j, j \ge 1} \rbrace$ were pseudo-random numbers from $Exp(1/5)$ generated using $rexp$ function in R.
  \end{itemize}
  
     For $a = 10$, $a_1 = 10$, $a_2 = 5$, $a_3 = 20/3$, and $a_4 = 10$.
   For $a = 5$, $a_1 = 5$, $a_2 = 2.5$, $a_3 = 10/3$, and $a_4 = 5$. 
     For $a = 2$, $a_1 = 2$, $a_2 = 1$, $a_3 = 4/3$, and $a_4 = 2$. 
 \item[(c)] The marginal distribution of $X_j$ was taken as ${F}_3(x) = 1-  e^{-x^{a}}, x \ge 0$, $a > 0$, $j \ge 1$.  We took  $a = 1.1, 1.2, 1.3$. The samples $\lbrace{X_j, 1 \le j \le n}\rbrace$ were generated as follows.
\begin{itemize}
\item[(S13)] $(m = 2)$  $X_j = min(X_j, X_{j+1})^{1/a}$,  where $\lbrace {X_j, j \ge 1} \rbrace$ were pseudo-random numbers from $Exp(1/2)$ generated using $rexp$ function in R.
 \item[(S14)]  $(m = 3)$  $X_j =  min( X_j, X_{j+1}, X_{j+2})^{1/a}$, where $\lbrace {X_j, j \ge 1} \rbrace$ were pseudo-random numbers from $Exp(1/3)$ generated using $rexp$ function in R.
  \item[(S15)]  $(m = 5)$  $X_j =  min( X_j, X_{j+1}, \cdots, X_{j+4})^{1/a}$, where $\lbrace {X_j, j \ge 1} \rbrace$ were pseudo-random numbers from $Exp(1/5)$ generated using $rexp$ function in R.
 \item [(S16)]  $(m = 10)$  $X_j =  min(X_j, \cdots, X_{j+9})^{1/a}$, where $\lbrace {X_j, j \ge 1} \rbrace$ were pseudo-random numbers from  $Exp(1/10)$ generated using $rexp$ function in R.
\end{itemize}

\end{itemize}
 \item[(3)]  The results are based on $r = 10,000$ replications and {\boldmath $\alpha = 0.05$}. 
\item[(4)] We chose $b = 0.5$ for Deshpande's test.
\item [(5)] Estimation of $\sigma_D$/ $\sigma_{HP}$/ $\sigma_A$: For the estimation of $\sigma_D$, $\sigma_{HP}$ and $\sigma_A$, we did not directly use the estimator $B_n$ given in Lemma $\ref{lemma64}$ as in practical applications the distribution function of the underlying population $F$ will be unknown. We therefore obtained the following result to get another consistent estimator $\hat{B}_n$ for the standard deviations. The estimator is based on the empirical (histogram) distribution function of the underlying sample. Proof of the following is in section $\ref{proofs}$.

\begin{theorem}\label{theorem41}
 Let $F_n(x)$ is the empirical (histogram) distribution function for $\{X_j, 1 \le j \le n\}$, and  $P(|X_j| \le C_1) = 1$, for some $0 < C_1 < \infty$, $j \ge 1$. Let $\hat{B}_n$  be analogous to $B_n$ with $S_j (k)$ replaced by $\hat{S}_j(k) = \sum_{i = j+1}^{j+ k}\hat{\rho}_1(X_i)$, and $\bar{X}_n$ by $\bar{\hat{X}}_n = \sum_{i =1}^n{\hat{\rho}_1(X_i)}/n$, where
\begin{itemize}
\item[(i)] for Deshpande's test, $\hat{\rho}_1(x) = \frac{F_n(x/b) + 1 - F_n(xb)}{2}$. 
\item[(ii)] for Hollander and Proschan's test, $\hat{\rho}_1(x) =  \frac{1}{3}\Big( \sum_{i: X_i \le x} F_n(x-X_i)/n + \sum_{i=1}^n\bar{F}_n(x+X_i)/n + \sum_{i: X_1 \ge x} F_n(X_i-x)/n \Big)$.  
\item[(iii)] for Ahmad's test, $\hat{\rho}_1(x) = 2x\bar{F}_n(x) - \frac{x}{2} + \frac{3\bar{X}_n}{2} - 2 \frac{\sum_{i =1}^n X_i I(X_i > x)}{n}$. 
\end{itemize}
 Let
\begin{equation} \label{equation41}
\sum_{j=n+1}^{\infty}Cov(X_1, X_j) = O(n^{-(s-2)/2}), \text{ for some $s > 6$}.
\end{equation}
Then $(\ref{equation41})$ is  sufficient to prove $|B_n - \hat{B}_n|$ $\to 0$ a.s as $n \to \infty$ for the 3 tests discussed. 
 We denote the standard deviation estimator $\hat{B}_n$ obtained using the above theorem as $\hat{\sigma}_D$, $\hat{\sigma}_{HP}$, and $\hat{\sigma}_A$  for Deshpande's, Hollander and Proschan's and Ahmad's test respectively.
\end{theorem}

We chose $\ell_n = [n^{1/3}]$, smallest integer less than or equal to $n^{1/3}$. Under the conditions assumed for obtaining the limiting distribution, $\rho_1$ is a lipshitz function for all the 3 tests discussed. In  Lemma $\ref{lemma64}$, $Y_i = \rho_1(X_i)$ and $\tilde{Y}_i = CX_i$, for some constant $C > 0$, for all $i \ge 1$. Under $\sum_{j=1}^{\infty} Cov(X_1, X_j) < \infty$, the estimator $B_n$ is consistent. Hence, from  condition $(\ref{equation41})$, $\hat{\sigma}_D$, $\hat{\sigma}_{HP}$, and $\hat{\sigma}_A$  are consistent estimators for $\sigma_D$, $\sigma_{HP}$ and $\sigma_A$ respectively.
\end{itemize}

\subsection{Simulation Results and Observations}

\begin{itemize}
 \item[(i)] \emph{Estimation of $\sigma_D$, $\sigma_{HP}$ and $\sigma_A$}: As discussed earlier, we used estimators  $\hat{\sigma}_D$, $\hat{\sigma}_{HP}$ and $\hat{\sigma}_A$ for simulations. For the sample generated from $Exp (1)$ ($\bar{F}(x) = e^{-x}$), using $(S1)$, $(S2)$, $(S3)$, and  $(S4)$, we analyzed the performance of the estimators by comparing them with the actual values $\sigma_D$, $\sigma_{HP}$ and $\sigma_A$ respectively. The simulation results given in Tables $4.1 (a) - (c)$ show that for a fixed $m$, as the sample size increases, the value of bias  and  the $E.M.S.E$ (Estimated M.S.E) of the estimator reduces. For $m =2, 3$, the convergence  is faster than for $m = 5, 10$, i.e a greater dependence leads to a slower convergence. 
   \item[(ii)] \emph{Asymptotic Normality}: From Table $4.2$, we observe that for a fixed $m$ as the sample size increases, the empirical size gets closer to $0.05$. For $m =10$,  larger sample sizes are  needed for a viable use of the asymptotic normality results than for $m=2, 3$.  The use of  estimators  for the standard deviations could also affect the convergence as the  bias and $E.M.S.E$ (Estimated M.S.E)  reduce much faster for $m = 2, 3$ than for $m = 5, 10$. 
\end{itemize}

\begin{center}
{\tiny\begin{tabular}{  p{.5in} p{.5in} p{.5in} p{.5in}| }
 \multicolumn{4}{c}{\textbf{ Table 4.1(a)}\: \emph{\textbf{Results for Deshpande's (D) test}}}\\
 \cline{1-4}

 \multicolumn{1}{|c|}{  (S1) (m=2), $2\sigma_D = 0.1778$} & {n=100}  &{n=200} & {n=500}   \\ \hline  
 \multicolumn{1}{|c|}{Bias = 2$|\sqrt{{\pi}/{2}}  \bar{\hat{\sigma}}_D - \sigma_D|$  }  & 0.0051 & 0.0047 & 0.0035 \\
\multicolumn{1}{|c|}{$2\sqrt{{\pi}/{2}}\bar{\hat{\sigma}}_D$  }   & 0.1727 & 0.1731 & 0.1743         \\
\multicolumn{1}{|c|}{E.M.S.E (2$\sqrt{{\pi}/{2}}\bar{\hat{\sigma}}_D$)}   &      0.0020 & 0.0013 & 0.0007      \\\hline
 \multicolumn{1}{|c|}{ (S2) (m=3), $2\sigma_D = 0.2155$} & {n=100}  &{n=200} & {n=500}   \\ \hline
 \multicolumn{1}{|c|}{Bias = 2$|\sqrt{{\pi}/{2}}  \bar{\hat{\sigma}}_D - \sigma_D|$  }   & 0.0078 &      0.0068  & 0.0049  \\
\multicolumn{1}{|c|}{$2\sqrt{{\pi}/{2}}\bar{\hat{\sigma}}_D$  }   &   0.2077 &      0.2087   & 0.2106        \\
\multicolumn{1}{|c|}{E.M.S.E (2$\sqrt{{\pi}/{2}}\bar{\hat{\sigma}}_D$)}    &   0.0033 &    0.0021  & 0.0011        \\ \hline
 \multicolumn{1}{|c|}{ (S3) (m=5), $2\sigma_D =  0.2767$}& {n=100}  &{n=200} & {n=500}  \\ \hline
 \multicolumn{1}{|c|}{Bias = 2$|\sqrt{{\pi}/{2}}  \bar{\hat{\sigma}}_D - \sigma_D|$  }   &  0.0133 & 0.0100 & 0.0072    \\
\multicolumn{1}{|c|}{$2\sqrt{{\pi}/{2}}\bar{\hat{\sigma}}_D$  }   &       0.2633 &  0.2667  & 0.2694       \\
\multicolumn{1}{|c|}{E.M.S.E (2$\sqrt{{\pi}/{2}}\bar{\hat{\sigma}}_D$)}   & 0.0067 &      0.0042  & 0.0022    \\\hline
 \multicolumn{1}{|c|}{ (S4) (m=10), $2\sigma_D =  0.3903$}& {n=100}  &{n=200} & {n=500} \\ \hline
 \multicolumn{1}{|c|}{Bias = 2$|\sqrt{{\pi}/{2}}  \bar{\hat{\sigma}}_D - \sigma_D|$  }    & 0.0440 &   0.0273          &  0.0164  \\
\multicolumn{1}{|c|}{$2\sqrt{{\pi}/{2}}\bar{\hat{\sigma}}_D$  }   &     0.3463 &    0.3631 &  0.3739       \\
\multicolumn{1}{|c|}{E.M.S.E (2$\sqrt{{\pi}/{2}}\bar{\hat{\sigma}}_D$)}   &   0.0158 &     0.0120 & 0.0064     \\\hline
 \end{tabular}}
 \end{center}

 \begin{center}
{\tiny\begin{tabular}{  p{.5in} p{.5in} p{.5in}p{.5in} | }
 \multicolumn{4}{c}{\textbf{ Table 4.1(b)}\: \emph{\textbf{Results for Hollander and Proschan's (HP) test}}}\\
 \cline{1-4}

 \multicolumn{1}{|c|}{  (S1) (m=2), $3\sigma_{HP}= 0.1438$} & {n=100}  &{n=200} & {n=500}    \\ \hline
  \multicolumn{1}{|c|}{Bias = 3$|\sqrt{{\pi}/{2}}  \bar{\hat{\sigma}}_{HP} - \sigma_{HP}|$  } &    0.0050 & 0.0043  &0.0031 \\
\multicolumn{1}{|c|}{$3\sqrt{{\pi}/{2}}\bar{\hat{\sigma}}_{HP}$  }   & 0.1388   & 0.1396 & 0.1407    \\
\multicolumn{1}{|c|}{E.M.S.E (3$\sqrt{{\pi}/{2}}\bar{\hat{\sigma}}_{HP}$)}  & 0.0014   & 0.0009 & 0.0005   \\ \hline

 \multicolumn{1}{|c|}{  (S2) (m=3), $3\sigma_{HP}=0.1741$} & {n=100}  &{n=200} & {n=500}  \\ \hline
  \multicolumn{1}{|c|}{Bias = 3$|\sqrt{{\pi}/{2}}  \bar{\hat{\sigma}}_{HP} - \sigma_{HP}|$  }  & 0.0074 & 0.0063 & 0.0044   \\
\multicolumn{1}{|c|}{$3\sqrt{{\pi}/{2}}\bar{\hat{\sigma}}_{HP}$  }   &  0.1667  & 0.1678 & 0.1697 \\\multicolumn{1}{|c|}{E.M.S.E (3$\sqrt{{\pi}/{2}}\bar{\hat{\sigma}}_{HP}$)}&  0.0022  & 0.0014 &  0.0008 \\ \hline

 \multicolumn{1}{|c|}{  (S3) (m=5), $3\sigma_{HP}= 0.2233$}& {n=100}  &{n=200} & {n=500}   \\ \hline
  \multicolumn{1}{|c|}{Bias = 3$|\sqrt{{\pi}/{2}}  \bar{\hat{\sigma}}_{HP} - \sigma_{HP}|$  } &  0.0122 & 0.0099 & 0.0068   \\
\multicolumn{1}{|c|}{$3\sqrt{{\pi}/{2}}\bar{\hat{\sigma}}_{HP}$  } & 0.2111   & 0.2134 & 0.2165    \\
\multicolumn{1}{|c|}{E.M.S.E (3$\sqrt{{\pi}/{2}}\bar{\hat{\sigma}}_{HP}$)}   &    0.0043 & 0.0028 & 0.0015  \\ \hline

 \multicolumn{1}{|c|}{  (S4) (m=10), $3\sigma_{HP}= 0.3150$} & {n=100}  &{n=200} & {n=500}  \\ \hline
  \multicolumn{1}{|c|}{Bias = 3$|\sqrt{{\pi}/{2}}  \bar{\hat{\sigma}}_{HP} - \sigma_{HP}|$  }  &  0.0354 & 0.0236 & 0.0144   \\
\multicolumn{1}{|c|}{$3\sqrt{{\pi}/{2}}\bar{\hat{\sigma}}_{HP}$  }  &  0.2796 & 0.2914 & 0.3006     \\
\multicolumn{1}{|c|}{E.M.S.E (3$\sqrt{{\pi}/{2}}\bar{\hat{\sigma}}_{HP}$)}   & 0.0097   & 0.0074 & 0.0041 \\ \hline
 \end{tabular}}
 \end{center}

 \begin{center}
{\tiny\begin{tabular}{  p{.5in} p{.5in} p{.5in} p{.5in}| }
 \multicolumn{4}{c}{\textbf{ Table 4.1(c)}\: \emph{\textbf{Results for Ahmad's (A) test}}}\\
 \cline{1-4}
 \multicolumn{1}{|c|}{  (S1) (m=2), $2\sigma_A = 0.7368$}  & {n=100}  &{n=200} & {n=500}   \\ \hline  
  \multicolumn{1}{|c|}{Bias = 2$|\sqrt{{\pi}/{2}}  \bar{\hat{\sigma}}_A - \sigma_A|$  } & 0.0492 & 0.0369 & 0.0243  \\
\multicolumn{1}{|c|}{$2\sqrt{{\pi}/{2}}\bar{\hat{\sigma}}_A$  } & 0.6876   &   0.6999       &    0.7125    \\
\multicolumn{1}{|c|}{E.M.S.E (2$\sqrt{{\pi}/{2}}\bar{\hat{\sigma}}_A$)}   & 0.0203  &    0.0145       &  0.0088  \\\hline
 \multicolumn{1}{|c|}{ (S2) (m=3), $2\sigma_A = 0.8803$} & {n=100}  &{n=200} & {n=500}   \\ \hline
 
  \multicolumn{1}{|c|}{Bias = 2$|\sqrt{{\pi}/{2}}  \bar{\hat{\sigma}}_A - \sigma_A|$  }   & 0.0671 & 0.0512 & 0.0341   \\
\multicolumn{1}{|c|}{$2\sqrt{{\pi}/{2}}\bar{\hat{\sigma}}_A$  }   & 0.8132 & 0.8291          &   0.8462   \\
\multicolumn{1}{|c|}{E.M.S.E (2$\sqrt{{\pi}/{2}}\bar{\hat{\sigma}}_A$)}   &  0.0293 &  0.0208     &     0.0128  \\ \hline

 \multicolumn{1}{|c|}{ (S3) (m=5), $2\sigma_A =  1.1209$}  & {n=100}  &{n=200} & {n=500}    \\ \hline
  \multicolumn{1}{|c|}{Bias = 2$|\sqrt{{\pi}/{2}}  \bar{\hat{\sigma}}_A - \sigma_A|$  }  & 0.1035  & 0.0798 & 0.0545  \\
\multicolumn{1}{|c|}{E.M.S.E (2$\sqrt{{\pi}/{2}}\bar{\hat{\sigma}}_A$)} &  1.0174 &   1.0411        & 1.0664       \\
\multicolumn{1}{|c|}{E.M.S.E (2$\sqrt{{\pi}/{2}}{\hat{B}}_n$)}& 0.0496  & 0.0352           &  0.0212  \\\hline
 \multicolumn{1}{|c|}{ (S4) (m=10), $2\sigma_A = 1.5757$}  & {n=100}  &{n=200} & {n=500}    \\ \hline
  \multicolumn{1}{|c|}{Bias = 2$|\sqrt{{\pi}/{2}}  \bar{\hat{\sigma}}_A - \sigma_A|$  }   &0.2673 & 0.1705 & 0.1069     \\
\multicolumn{1}{|c|}{$2\sqrt{{\pi}/{2}}\bar{\hat{\sigma}}_A$  }  &1.3084  & 1.4052 & 1.4687    \\
\multicolumn{1}{|c|}{E.M.S.E (2$\sqrt{{\pi}/{2}}\bar{\hat{\sigma}}_A$)}  & 0.1054 & 0.0702 & 0.0438  \\\hline
 \end{tabular}}
 \end{center}

 {\scriptsize\ In Tables $4.1 (a)\Big[ (b)\Big] \Big((c)\Big)$, $\bar{\hat{\sigma}}_D = \frac{1}{r}\sum_{i=1}^r \hat{\sigma}_D(i)$ $\Big[\bar{\hat{\sigma}}_{H.P} = \frac{1}{r}\sum_{i=1}^r \hat{\sigma}_{HP}(i)\Big]$ $\Big(\bar{\hat{\sigma}}_A = \frac{1}{r}\sum_{i=1}^r \hat{\sigma}_A(i)\Big)$. \\$E.M.S.E (\bar{\hat{\sigma}}_D) = \frac{1}{r -1} \sum_{i=1}^{r}(\hat{\sigma}_D(i) - \bar{\hat{\sigma}}_D)^2$ $\Big[E.M.S.E (\bar{\hat{\sigma}}_{HP}) = \frac{1}{r -1} \sum_{i=1}^{r}(\hat{\sigma}_{HP}(i) - \bar{\hat{\sigma}}_{HP})^2\Big]$ \\$\Big( E.M.S.E (\bar{\hat{\sigma}}_A) = \frac{1}{r -1} \sum_{i=1}^{r}(\hat{\sigma}_A(i) - \bar{\hat{\sigma}}_A)^2\Big)$, where $\hat{\sigma}_D(i) \Big[\hat{\sigma}_{HP}(i)\Big] \Big(\hat{\sigma}_A(i)\Big)$, $1 \le i \le r$, denote the estimated value for each sample for the $D \Big[HP\Big] \Big(A\Big)$ tests calculated using Comment $(5)$.}\\

\begin{center}
{\tiny\begin{tabular}{ p{.2.5in} p{1.2in}  p{1.2in}  p{1.2in}  | }
 \multicolumn{4}{c}{\textbf{ Table 4.2}\: \emph{\textbf{Simulation Results for $\bar{F}(x) = e^{-x}$  $\Big(D$ $\Big[HP\Big]$ $\Big( A\Big)$$\Big)$}}}\\
 \cline{1-4}

 \multicolumn{1}{|c|}{  (S1) (m=2)} &  \multicolumn{1}{|c}{n=100}  & \multicolumn{1}{c}{n=200} & \multicolumn{1}{c|}{n=500}   \\ \hline
\multicolumn{1}{|c|}{Sim. size}  & 0.0642 $\Big[0.0657\Big]$ $\Big(0.0761\Big)$ &  0.0601$\Big[ 0.0567\Big]$ $\Big(0.0662\Big)$  & 0.0553$\Big[0.0550\Big]$$\Big(0.0628\Big)$  \\
\multicolumn{1}{|c|}{Sim. critpt}  & 1.7991$\Big[-1.7874\Big]$$\Big(1.9226\Big)$ & 1.7530$\Big[-1.7251\Big]$$\Big(1.8066\Big)$ &  1.6995$\Big[-1.6914\Big]$$\Big(1.7756\Big)$\\ \hline

\multicolumn{1}{|c|}{Sim. size (if assumed $i.i.d.$) }  & 0.1488$\Big[0.1533\Big]$$\Big(0.1166\Big)$  &  0.1354$\Big[ 0.1383\Big]$$\Big(0.11\Big)$  & 0.1268$\Big[0.1269\Big]$$\Big(0.1041\Big)$) \\
\multicolumn{1}{|c|}{Sim. critpt (if assumed $i.i.d.$) } &2.6602$\Big[-2.6778\Big]$$\Big(2.2396\Big)$ & 2.5080$\Big[-2.4975\Big]$$\Big(2.1354\Big)$  & 2.3628$\Big[-2.4021\Big]$$\Big(2.1488\Big)$ \\ \hline

  \multicolumn{1}{|c|}{ (S2) (m=3)}& \multicolumn{1}{|c}{n=100}  & \multicolumn{1}{c}{n=200} & \multicolumn{1}{c|}{n=500}   \\ \hline
\multicolumn{1}{|c|}{Sim. size }  &  0.0791$\Big[0.0756\Big]$$\Big(0.0806\Big)$  & 0.0689$\Big[0.0661\Big]$$\Big(0.0725\Big)$   & 0.0595$\Big[0.058\Big]$$\Big(0.0639\Big)$  \\
\multicolumn{1}{|c|}{Sim. critpt }  & 1.8992$\Big[-1.8848\Big]$$\Big(1.9196\Big)$  & 1.7967$\Big[-1.7958\Big]$$\Big(1.8487\Big)$   &  1.7236$\Big[-1.7175\Big]$$\Big(1.7714\Big)$ \\ \hline

\multicolumn{1}{|c|}{Sim. size (if assumed $i.i.d.$) }  &  0.2198$\Big[0.2271\Big]$$\Big(0.1709\Big)$  & 0.1984$\Big[0.206\Big]$$\Big(0.16\Big)$  & 0.1853$\Big[0.1869\Big]$$\Big(0.1532\Big)$  \\
\multicolumn{1}{|c|}{Sim. critpt (if assumed $i.i.d.$) } & 3.3864$\Big[-3.4649\Big]$$\Big(2.7259\Big)$ & 3.1938$\Big[-3.1852\Big]$$\Big(2.6425\Big)$ & 2.9841$\Big[-3.0180\Big]$$\Big(2.5802\Big)$  \\ \hline

 \multicolumn{1}{|c|}{ (S3) (m=5)} & \multicolumn{1}{|c}{n=100}  & \multicolumn{1}{c}{n=200} & \multicolumn{1}{c|}{n=500}   \\ \hline

\multicolumn{1}{|c|}{Sim. size }  &   0.1046$\Big[0.1079\Big]$$\Big(0.0957\Big)$  & 0.0848$\Big[ 0.0862\Big]$$\Big(0.0803\Big)$  & 0.0672$\Big[0.0681\Big]$$\Big(0.0680\Big)$   \\
\multicolumn{1}{|c|}{Sim. critpt }  &  2.1295$\Big[-2.1216\Big]$$\Big(2.0860\Big)$  & 1.9417$\Big[-1.9394\Big]$$\Big(1.9259\Big)$ &  1.7870$\Big[-1.8009\Big]$$\Big(1.8079\Big)$ \\ \hline

\multicolumn{1}{|c|}{Sim. size (if assumed $i.i.d.$) }  & 0.3253$\Big[0.3470\Big]$$\Big(0.2608\Big)$ &  0.2877$\Big[0.3033\Big]$$\Big(0.2403\Big)$ & 0.2605$\Big[ 0.2709\Big]$$\Big(0.2247\Big)$ \\
\multicolumn{1}{|c|}{Sim. critpt (if assumed $i.i.d.$) } &4.7930$\Big[-4.8668\Big]$$\Big(3.5757\Big)$ & 4.4227$\Big[-4.5284\Big]$$\Big(3.4740\Big)$ & 4.0286$\Big[-4.1122\Big]$$\Big(3.3397\Big)$ \\ \hline
 \multicolumn{1}{|c|}{ (S4) (m=10)} & \multicolumn{1}{|c}{n=100}  & \multicolumn{1}{c}{n=200} & \multicolumn{1}{c|}{n=500}   \\ \hline
\multicolumn{1}{|c|}{Sim. size }  & 0.1756 $\Big[0.1804\Big]$$\Big(0.1538\Big)$ &  0.1207$\Big[0.1238\Big]$$\Big(0.1103\Big)$  & 0.0836$\Big[0.0817\Big]$$\Big(0.0826\Big)$ \\
\multicolumn{1}{|c|}{Sim. critpt }  & 2.7790$\Big[-2.6765\Big]$$\Big(2.6487\Big)$ & 2.2874$\Big[-2.1869\Big]$$\Big(2.1926\Big)$ &  1.9331$\Big[-1.9052\Big]$$\Big(1.9315\Big)$  \\ \hline

\multicolumn{1}{|c|}{Sim. size (if assumed $i.i.d.$)}  &  0.4786$\Big[0.5333\Big]$$\Big(0.4044\Big)$ & 0.4183$\Big[ 0.4685\Big]$$\Big(0.3572\Big)$  &  0.3662$\Big[0.3932\Big]$$\Big(0.3252\Big)$  \\
\multicolumn{1}{|c|}{Sim. critpt (if assumed $i.i.d.$)} &7.9272$\Big[-8.1142\Big]$$\Big(5.4121\Big)$  & 7.1790$\Big[-7.1295\Big]$$\Big(5.2703\Big)$ &  6.3136$\Big[-6.2565\Big]$$\Big(4.9758\Big)$ \\ \hline
 \end{tabular}}
\end{center}

{\scriptsize In Table $4.2$, Sim. critpt gives the simulated critical point, and Sim. size gives the simulated size of the test. The simulated critical point is the $95^{th}$ $\Big[5^{th}\Big]$ $\Big(95^{th}\Big)$ percentile of the generated $r = 10,000$ standardized statistic values.The 
standardized statistic values are given by, $\frac{\sqrt{n}(J_{(n, b)}(i) - 1/(b+1))}{2\hat{\sigma}_D(i)}$ for Despande's test (D),
$\Big[\frac{\sqrt{n}(N_n(i) - 1/4))}{3\hat{\sigma}_{HP}(i)} \text{ for  }$ \\ $\text{Hollander and Proschan's test}$ $\text{(HP)}\Big]$
$\Big(\frac{\sqrt{n}\delta_{F_n}(i)}{2\hat{\sigma}_A(i)} \text{ for Ahmad's test (A)}\Big)$, where $J_{(n, b)}(i)$ $\Big[N_n(i)\Big]$ $\Big(\delta_{F_n}(i)\Big)$ denote the sample statistic values,  and ${\hat{\sigma}}_D(i)$ $\Big[{\hat{\sigma}}_{HP}(i)\Big]$ $\Big({\hat{\sigma}}_A(i)\Big)$ denote the estimated value for each sample for the D$\Big[ HP\Big]$ $\Big(A\Big)$ tests for the $i^{th}$ replication, $1 \le i \le r$. The  simulated size of the test is the number of generated standardized statistic values greater$\Big[\text{less}\Big]$ $\Big(\text{greater}\Big)$ than $95^{th}$ $\Big[5^{th}\Big]$ $\Big(95^{th}\Big)$ percentile of the standard normal distribution given by $z_{0.95} \Big[z_{0.05}\Big] \Big(z_{0.95}\Big)$, where $z_{0.95} =1.644854$ and $z_{0.05} = -z_{0.95}$. The Sim. critpt (if assumed $i.i.d.$) and Sim. size (if assumed $i.i.d.$) is calculated  under the wrong assumption of independence for the same samples and is based on the standardized statistic values given by $(3.2)$ $\Big[(3.5)\Big]$$\Big((3.10)\Big)$. }

\begin{itemize}

 \item[(iii)] \emph{Effect of wrongly assuming the associated observations to be $i.i.d.$:} From Table $4.2$, we observe that wrongly assuming an associated sequence to be $i.i.d.$ leads to the estimated size of the test being farther away from 0.05, than in comparison with correctly considering the associated case. This is expected as the covariance terms are excluded under the false assumption of independence. For example, for Deshpande's test we observe that for a sample of size 500 and $m = 2$ wrongly considering the observations to be $i.i.d.$ leads to the simulated size of the test being 0.1268, much greater than the observed size of 0.0553 obtained under the correct assumption of association. This discrepancy can be observed more when $m = 10$.
 
 \item[(iv)] \emph{Power of the test:} From the following Tables $4.3 - 4.5$, we observe that the empirical power of the test is lower in the case when the sample, under $H_1$,  is generated from a distribution which is closer to $F(x)$. For example, in Table $4.3$, $F_1(x) = 1- e^{-(e^{ax} -1)/a}$ is closer to $F(x)$ when $a = 0.5$ than when $a = 1$ and hence, the power of the test increases as $a$ moves closer to 1. The empirical power increases as the sample size increases. In general, for the same sample size, a greater order of dependence leads to a reduction in the power of the test.

 \begin{center}
{\tiny\begin{tabular}{ p{.3in} p{1in}  p{1in}  p{1in}| }
 \multicolumn{4}{c}{\textbf{ Table 4.3}\: \emph{\textbf{Simulation Results for power $\bar{F}_1(x) = e^{-(e^{ax} -1)/a}$ $\Big(D$ $\Big[HP\Big]$ $\Big( A\Big)$$\Big)$ }}}\\
 \cline{1-4}
 \multicolumn{1}{|c|}{ (S5) (m=2)} &\multicolumn{1}{|c}{n=100}  & \multicolumn{1}{c}{n=200} & \multicolumn{1}{c|}{n=500}   \\ \hline
\multicolumn{1}{|c|}{ $a = 0.5$} & 0.4477$\Big[0.4723\Big]$$\Big(0.8948\Big)$ &0.6587$\Big[0.6944\Big]$$\Big(0.9839\Big)$ & 0.9432$\Big[0.9601\Big]$$\Big(0.9999\Big)$ \\
\multicolumn{1}{|c|}{ $a = 0.8$}& 0.6402$\Big[0.6786\Big]$$\Big(0.9719\Big)$ & 0.8775$\Big[0.9027\Big]$$\Big(0.9987\Big)$ & 0.9972$\Big[0.9990\Big]$$\Big(1\Big)$\\
\multicolumn{1}{|c|}{ $a = 1$}&    0.7421$\Big[0.7700\Big]$$\Big(0.9880\Big)$ & 0.9373$\Big[0.9560\Big]$$\Big(1\Big)$ &  0.9997$\Big[0.9999\Big]$$\Big(1\Big)$\\ \hline
 \multicolumn{1}{|c|}{ (S6) (m=3)}&\multicolumn{1}{|c}{n=100}  & \multicolumn{1}{c}{n=200} & \multicolumn{1}{c|}{n=500}   \\ \hline
\multicolumn{1}{|c|}{ $a = 0.5$} & 0.3747$\Big[0.3999\Big]$$\Big(0.8344\Big)$  & 0.5397$\Big[0.5739\Big]$$\Big(0.9475\Big)$  & 0.8572$\Big[0.8879\Big]$$\Big(0.9988\Big)$ \\
\multicolumn{1}{|c|}{ $a = 0.8$}&  0.5362$\Big[0.5696\Big]$$\Big(0.9308\Big)$ & 0.7608$\Big[0.7970\Big]$$\Big(0.9922\Big)$ & 0.9786$\Big[0.9874\Big]$$\Big(1\Big)$\\
\multicolumn{1}{|c|}{ $a = 1$}&  0.6232$\Big[0.6620\Big]$$\Big(0.9603\Big)$ & 0.8513$\Big[0.8841\Big]$$\Big(0.998\Big)$ & 0.9950$\Big[0.9975\Big]$$\Big(1\Big)$ \\ \hline
 \multicolumn{1}{|c|}{  (S7) (m=5)} & \multicolumn{1}{|c}{n=100}  & \multicolumn{1}{c}{n=200} & \multicolumn{1}{c|}{n=500}   \\ \hline
\multicolumn{1}{|c|}{ $a = 0.5$} & 0.3286$\Big[0.3527\Big]$$\Big(0.7448\Big)$   & 0.4196$\Big[0.4548\Big]$$\Big(0.8749\Big)$ &  0.6951$\Big[0.7392\Big]$$\Big(0.9872\Big)$ \\
\multicolumn{1}{|c|}{ $a = 0.8$}& 0.4442$\Big[0.4785\Big]$$\Big(0.8587\Big)$  & 0.5968$\Big[0.6428\Big]$$\Big(0.9606\Big)$ & 0.9005$\Big[0.9291\Big]$$\Big(0.9994\Big)$\\
\multicolumn{1}{|c|}{ $a = 1$}&  0.5076$\Big[0.5463\Big]$$\Big(0.8973\Big)$ & 0.6917$\Big[0.7388\Big]$$\Big(0.9818\Big)$ & 0.9574$\Big[0.9717\Big]$$\Big(0.9999\Big)$ \\ \hline
 \multicolumn{1}{|c|}{  (S8) (m=10)} & \multicolumn{1}{|c}{n=100}  & \multicolumn{1}{c}{n=200} & \multicolumn{1}{c|}{n=500}   \\ \hline
\multicolumn{1}{|c|}{ $a = 0.5$} &0.3549$\Big[0.3826\Big]$$\Big(0.6716\Big)$   &  0.3544$\Big[0.3887\Big]$$\Big(0.7631\Big)$ & 0.4939$\Big[0.5349\Big]$$\Big(0.9247\Big)$\\
\multicolumn{1}{|c|}{ $a = 0.8$}& 0.4360$\Big[0.4662\Big]$$\Big(0.7604\Big)$  & 0.4762$\Big[0.5224\Big]$$\Big(0.8697\Big)$  &  0.7000$\Big[0.7504\Big]$$\Big(0.9836\Big)$\\
\multicolumn{1}{|c|}{ $a = 1$}&  0.4746$\Big[0.5118\Big]$$\Big(0.8006\Big)$ & 0.5415$\Big[0.5857\Big]$$\Big(0.9104\Big)$ & 0.7907$\Big[0.8359\Big]$$\Big(0.9941\Big)$ \\ \hline
 \end{tabular}}
 \end{center}

   \begin{center}
{\tiny\begin{tabular}{ p{.3in} p{1in}  p{1in}  p{1in}| }
 \multicolumn{4}{c}{\textbf{ Table 4.4}\: \emph{\textbf{Simulation Results for power $\bar{F}_2(x) = e^{-x -\frac{x^{2}}{a}}$$\Big(D$ $\Big[HP\Big]$ $\Big( A\Big)$$\Big)$}}}\\
 \cline{1-4}
  \multicolumn{1}{|c|}{ (S9) (m=2)} & \multicolumn{1}{|c}{n=100}  & \multicolumn{1}{c}{n=200} & \multicolumn{1}{c|}{n=500}   \\ \hline
\multicolumn{1}{|c|}{ $a = 10$}   &0.2255$\Big[0.2295\Big]$$\Big(0.6650\Big)$  & 0.3410$\Big[0.3459\Big]$$\Big(0.8175\Big)$ &  0.6162$\Big[0.6389\Big]$$\Big(0.9686\Big)$ \\ 
\multicolumn{1}{|c|}{ $a = 5$}    & 0.3964$\Big[0.4025\Big]$$\Big(0.8369\Big)$ & 0.6022$\Big[0.6181\Big]$$\Big(0.9584\Big)$  &  0.9152$\Big[0.9265\Big]$$\Big(0.9995\Big)$  \\
\multicolumn{1}{|c|}{ $a = 2$}  &  0.6511$\Big[0.6655\Big]$$\Big(0.9655\Big)$  & 0.8900$\Big[0.9023\Big]$$\Big(0.9984\Big)$  &   0.9986$\Big[0.9991\Big]$$\Big(1\Big)$ \\ \hline 
 \multicolumn{1}{|c|}{ (S10) (m=3)}&\multicolumn{1}{|c}{n=100}  & \multicolumn{1}{c}{n=200} & \multicolumn{1}{c|}{n=500}   \\ \hline
 \multicolumn{1}{|c|}{ $a = 10$}   & 0.1936$\Big[0.1922\Big]$$\Big(0.5779\Big)$   & 0.2548$\Big[0.2599\Big]$$\Big(0.6926\Big)$  &   0.4188$\Big[0.4368\Big]$$\Big(0.8888\Big)$  \\ 
\multicolumn{1}{|c|}{ $a = 5$}   &  0.2902$\Big[0.2909\Big]$$\Big(0.7300\Big)$ & 0.4145$\Big[0.4232\Big]$$\Big(0.8675\Big)$ & 0.7099$\Big[0.7356\Big]$$\Big(0.9828\Big)$   \\
\multicolumn{1}{|c|}{ $a = 2$}  &    0.4730$\Big[0.4916\Big]$$\Big(0.8944\Big)$ &  0.6976$\Big[0.7194\Big]$$\Big(0.9806\Big)$ &  0.9612$\Big[0.9694\Big]$$\Big(0.9999\Big)$  \\ \hline
 \multicolumn{1}{|c|}{  (S11) (m=5)} & \multicolumn{1}{|c}{n=100}  & \multicolumn{1}{c}{n=200} & \multicolumn{1}{c|}{n=500}   \\ \hline
\multicolumn{1}{|c|}{ $a = 10$}   &  0.2008$\Big[0.1999\Big]$$\Big(0.5501\Big)$ & 0.2350$\Big[0.2358\Big]$$\Big(0.6416\Big)$  & 0.3538$\Big[0.3688\Big]$$\Big(0.8319\Big)$   \\ 
\multicolumn{1}{|c|}{ $a = 5$}   &  0.2727$\Big[0.2721\Big]$$\Big( 0.6709\Big)$  & 0.3558$\Big[0.3631\Big]$$\Big(0.7961\Big)$  & 0.5883$\Big[0.6119\Big]$$\Big(0.9549\Big)$   \\
\multicolumn{1}{|c|}{ $a = 2$}  &  0.4001$\Big[0.4173\Big]$$\Big(0.8286\Big)$ & 0.5640$\Big[0.5845\Big]$$\Big(0.9416\Big)$   & 0.8799$\Big[0.9011\Big]$$\Big( 0.9982\Big)$   \\ \hline
 \multicolumn{1}{|c|}{ (S12) (m=10)} & \multicolumn{1}{|c}{n=100}  & \multicolumn{1}{c}{n=200} & \multicolumn{1}{c|}{n=500}   \\ \hline
\multicolumn{1}{|c|}{ $a = 10$}    & 0.2354$\Big[0.2431\Big]$$\Big(0.5390\Big)$  & 0.2376$\Big[0.2450\Big]$$\Big(0.5880\Big)$  & 0.3027$\Big[0.3198\Big]$$\Big(0.7513\Big)$ \\ 
\multicolumn{1}{|c|}{ $a = 5$}   &  0.3148$\Big[0.3217\Big]$$\Big(0.6381\Big)$  & 0.3395$\Big[0.3476\Big]$$\Big(0.7268\Big)$  & 0.4856$\Big[0.5078\Big]$$\Big(0.8948\Big)$   \\
\multicolumn{1}{|c|}{ $a = 2$}  & 0.4058$\Big[0.4221\Big]$$\Big(0.7531\Big)$  & 0.4791$\Big[0.4978\Big]$$\Big(0.8694\Big)$  &  0.7313$\Big[0.7542\Big]$$\Big(0.9835\Big)$  \\ \hline
 \end{tabular}}
 \end{center}

  \begin{center}
{\tiny\begin{tabular}{ p{.3in} p{1in}  p{1in}  p{1in}| }
 \multicolumn{4}{c}{\textbf{ Table 4.5}\: \emph{\textbf{Simulation Results for power  $\bar{F}_3(x) = e^{-x^{a}}$ $\Big(D$ $\Big[HP\Big]$ $\Big( A\Big)$$\Big)$)}}}\\
 \cline{1-4}
  \multicolumn{1}{|c|}{ (S13) (m=2)}  & \multicolumn{1}{|c}{n=100}  & \multicolumn{1}{c}{n=200} & \multicolumn{1}{c|}{n=500}   \\ \hline
\multicolumn{1}{|c|}{ $a = {1.1}$}   &0.2352$\Big[0.2366\Big]$$\Big(0.5847\Big)$  & 0.3287$\Big[0.3285\Big]$$\Big(0.7053\Big)$  &  0.5927$\Big[0.5898\Big]$$\Big(0.8929\Big)$  \\ 
\multicolumn{1}{|c|}{ $a = 1.2$}  &  0.4984$\Big[0.4977\Big]$$\Big(0.8467\Big)$ & 0.7227$\Big[0.7289\Big]$$\Big(0.9569\Big)$ & 0.9697$\Big[0.9708\Big]$$\Big(0.9994\Big)$ \\
\multicolumn{1}{|c|}{ $a = 1.3$} &    0.7379$\Big[0.7462\Big]$$\Big(0.9660\Big)$ & 0.9403$\Big[0.9467\Big]$$\Big(0.9981\Big)$  & 0.9998$\Big[0.9998\Big]$$\Big(1\Big)$\\ \hline
  \multicolumn{1}{|c|}{ (S14) (m=3)} & \multicolumn{1}{|c}{n=100}  & \multicolumn{1}{c}{n=200} & \multicolumn{1}{c|}{n=500}   \\ \hline
\multicolumn{1}{|c|}{ $a = {1.1}$}    & 0.2163$\Big[0.2187\Big]$$\Big(0.5449\Big)$    &  0.2771$\Big[0.2827\Big]$$\Big(0.6421\Big)$   &  0.4696$\Big[0.4763\Big]$$\Big(0.8312\Big)$\\ 
\multicolumn{1}{|c|}{ $a = 1.2$}  &  0.4175$\Big[0.4763\Big]$$\Big(0.7925\Big)$ & 0.5977$\Big[0.6116\Big]$$\Big(0.9125\Big)$  &   0.9050$\Big[0.9142\Big]$$\Big(0.9939\Big)$\\
\multicolumn{1}{|c|}{ $a = 1.3$} &   0.6236$\Big[0.6429\Big]$$\Big(0.9264\Big)$ & 0.8546$\Big[0.8666\Big]$$\Big(0.9887\Big)$  & 0.9959$\Big[0.9962\Big]$$\Big(0.9999\Big)$ \\ \hline
  \multicolumn{1}{|c|}{ (S15) (m=5)} & \multicolumn{1}{|c}{n=100}  & \multicolumn{1}{c}{n=200} & \multicolumn{1}{c|}{n=500}   \\ \hline
\multicolumn{1}{|c|}{ $a = {1.1}$}   & 0.2168$\Big[0.2255\Big]$$\Big(0.5031\Big)$ &  0.2391$\Big[0.2491\Big]$$\Big(0.5690\Big)$ & 0.3546$\Big[0.3674\Big]$$\Big(0.7348\Big)$  \\ 
\multicolumn{1}{|c|}{ $a = 1.2$}  &   0.3585$\Big[0.3765\Big]$$\Big(0.7111\Big)$ & 0.4653$\Big[0.4902\Big]$$\Big(0.8333\Big)$  & 0.7649$\Big[0.7768\Big]$$\Big(0.9707\Big)$   \\
\multicolumn{1}{|c|}{ $a = 1.3$} &  0.5144$\Big[0.5405\Big]$$\Big(0.8609\Big)$  &  0.7007$\Big[0.7246\Big]$$\Big(0.9551\Big)$  & 0.9608$\Big[0.9673\Big]$$\Big( 0.9990\Big)$\\ \hline
  \multicolumn{1}{|c|}{ (S16) (m=10)} & \multicolumn{1}{|c}{n=100}  & \multicolumn{1}{c}{n=200} & \multicolumn{1}{c|}{n=500}   \\ \hline
\multicolumn{1}{|c|}{ $a = {1.1}$}  & 0.2672$\Big[0.2851\Big]$$\Big(0.5158\Big)$  &0.2359$\Big[0.2472\Big]$$\Big(0.5229\Big)$  & 0.2693$\Big[0.2843\Big]$$\Big(0.6247\Big)$  \\ 
\multicolumn{1}{|c|}{ $a = 1.2$}  & 0.3722$\Big[0.3992\Big]$$\Big(0.6588\Big)$  & 0.3912$\Big[0.4157\Big]$$\Big(0.7325\Big)$   &  0.5517$\Big[0.5769\Big]$$\Big(0.8864\Big)$  \\
\multicolumn{1}{|c|}{ $a = 1.3$} & 0.4759$\Big[0.5110\Big]$$\Big(0.7759\Big)$  &0.5478$\Big[0.5834\Big]$$\Big(0.8715\Big)$  & 0.8024$\Big[0.8246\Big]$$\Big(0.9807\Big)$  \\ \hline

 \end{tabular}}
\end{center}

{\scriptsize\ In Tables $4.3 - 4.5$, Sim. power $= \frac{N}{r}$, where\\
 $N$ $=$ $\#$ $\lbrace{ i: \frac{\sqrt{n}(U_n(i) - 1/(b+1))}{2\hat{\sigma}_D(i)} \ge z_{0.95}}\rbrace$ for Despande's test (D),\\
 $N$ $=$ $\#$ $\lbrace{ i: \frac{\sqrt{n}(U_n(i) - 1/4)}{3\hat{\sigma}_{HP}(i)} \le -z_{0.95}}\rbrace$ for Hollander and Proschan's test (H.P.),\\
 $N$ $=$ $\#$ $\lbrace{ i: \frac{\sqrt{n}U_n(i) }{2\hat{\sigma}_A(i)} \ge z_{0.95}}\rbrace$ for Ahmad's test (A).}

 \item[(iv)] \emph{Comparison with the $i.i.d$ setup}: A comparison of the simulation results for the statistics done under the $i.i.d.$ setup, indicate that relatively larger sample sizes are needed for applying the asymptotic normality results under the dependent setup.  
\end{itemize}

\section{Discussion}\label{dis}

In this paper, we have discussed the limiting properties of tests by \cite{DESHPANDE01081983}, \cite{hollander1972testing} and \cite{ahmad1992} for testing exponentiality against IFRA, NBU and DMRL respectively, when the underlying random variables are stationary and associated.  Simulation results indicate that in comparison with the $i.i.d.$ setup relatively larger sample sizes are needed  for  use of normal distribution approximation. 

 Apart from the test statistics considered, the limiting distribution of other U-statistics with the discussed type of kernel can be obtained under the conditions of Theorem $\ref{theorem25}$. This paper also adds to the existing literature on U-statistics based on associated random variables. 
 
 The tests discussed above cannot be used to test $F \in \mathbb{F}$ against the alternative $F \notin \mathbb{F}$, where $\mathbb{F}$ is a family of distributions with some ageing property (IFRA, NBU, DMRL etc.). Recently, many authors have  proposed tests for membership of the proposed class (i.e $F \in \mathbb{F}$) against the alternative of non-membership of that class (i.e $F \notin \mathbb{F}$). For examples, see \cite{10.2307/3448683}, \cite{4585402} and \cite{srivastava2012testing}. Their tests  are for $i.i.d.$ setup. Extension of their results to the case when the underlying observations are associated are being looked into.

\section{Proofs} \label{proofs}
\subsection{Auxiliary Results}
\setlength\parindent{24pt}
     In this section we give results and definitions which will be needed to  prove our main results.

\begin{lemma}\label{lemma61}(\cite{MR789244})  Let $\left\lbrace X_n, n \ge 1\right\rbrace$  be a stationary sequence of associated random variables, with  $E(X_1^2) < \infty$. Then,
\begin{equation}
|\phi - \prod_{j=1}^{n} \phi_j| \le 2\sum_{1 \le k < l \le n} | r_k|| r_l| Cov (X_k, X_l ),
\end{equation}
where $\phi = E(exp(i \sum_{j=1}^n r_j X_j))$ and $\phi_j= E(exp(i r_j X_j))$, $j = 1, \cdots, n$  are joint and marginal characteristic functions, respectively.
\end{lemma}

\begin{lemma} \label{lemma62} (\cite{Roussas2001397}) Let ${X} = (X_1, ..., X_k)$ and ${X'} = (X'_1, ..., X'_k)$ be two $k$-dimensional random vectors with characteristic functions $\Phi_X$ and $\Phi_{X'}$ respectively. \\
$A1)$ The $p.d.f$s $f_{{X}}$ and $f_{{X'}}$ of  ${X}$ and  ${X'}$ are bounded and satisfy a Lipshitz condition of order 1.\\
$A2)$ the characteristic functions $\Phi_X$ and $\Phi_{X'}$ are absolutely integrable.\\
Under $A1$ and $A2$, and for any $T_j >0, j =1, ..., k$,
\begin{align}
sup\lbrace{|f_{{X}}(\textbf{x}) - f_{{X'}}(\textbf{x})|; \textbf{x} \in \mathbb{R}^k }\rbrace & \le \frac{1}{(2\pi)^k}\int_{-T_k}^{T_k} ... \int_{-T_1}^{T_1}|\Phi_{{X}}(\textbf{t}) - \Phi_{{X'}}(\textbf{t})|d\textbf{t}  + 4C\sqrt{3} \sum_{j=1}^k \frac{1}{T_j}
\end{align}
holds, where $C$ is an absolute constant.
\end{lemma}

\begin{lemma} \label{lemma64} (\cite{gargdewan2016}) Let $\left\lbrace X_n, n \ge 1\right\rbrace$  be a stationary sequence of associated random variables. For each j, let $Y_j = f(X_j)$ and $\tilde{Y_j} = \tilde{f}(X_j)$. Suppose that $f \ll \tilde{f}$.  Let  $\left\lbrace \ell_n, n \ge 1 \right\rbrace$   be a sequence of positive integers with $1\le \ell_n\le n$ and  $\ell_n$ $=$ $o(n)$ as $ n \to \infty$. Set $S_j(k) = \sum_{i= j +1}^{j+k}Y_i$ , $\bar{X_n} =\frac{1}{n} \sum_{j =1}^{n} Y_j$. Let $E(Y_1) = \mu$ and $E(Y_1^2) < \infty$. Define, $(write\hspace{2 mm} \ell = \ell_n)$,
\begin{equation}
B_n = \frac{1}{n-\ell+1} \Bigg(\sum_{j=0}^{n-\ell} \frac{| S_j(\ell) - \ell\bar{Y_n}|}{\sqrt{\ell}}\Bigg).
\end{equation}
Assume $\sum_{j=1}^{\infty}Cov(\tilde{Y_1},\tilde{Y_j}) < \infty$. Then,
\begin{equation}
B_n \to \sigma_f\sqrt{\frac{2}{\pi}}  \hspace{3 mm} \text{in}  \hspace{2 mm} L_2 \hspace{2 mm} \text{as} \hspace{2 mm}n \to \infty, \text{where $\sigma_f^2 = Var(Y_1) + 2\sum_{j = 2}^{\infty} Cov (Y_1, Y_j)$.}
\end{equation}

\end{lemma}

\begin{lemma}\label{roussaspaper} (\cite{roussas1993curve}) Let the sequence $\{X_n, n \ge 1\}$ be a stationary associated sequence of random variables with bounded one-dimensional probability density function. Suppose,
\begin{equation}
u(n) = 2 \sum_{j = n+1}^{\infty} Cov(X_1, X_j) = O(n^{-(s-2)/2}), \text{ for some $s > 2$}.
\end{equation}
Let $\psi_n$ be any positive norming factor. Then, for any bounded interval $[-C_1, C_1]$, we have,
\begin{equation*}
\underset{x \in [-C_1,C_1]}{sup} \psi_n|F_n(x) - F(x)| \to 0, \text{ a.s as $n \to \infty$, provided } \sum_{n =1}^{\infty} n^{-s/2} \psi_n^{s+2} < \infty.
\end{equation*}
\end{lemma}
\subsection{Proofs of main results}

The proof of Theorem $\ref{theorem25}$ requires the following result.
\begin{lemma}\label{pdflemma2}
Assume the density functions $f_{X_{i_1},  X_{i_2}, X_{i_3}, X_{i_4}}$ and $f_{X_{i_1}',  X_{i_2}, X_{i_3}, X_{i_4}}$ are bounded and satisfy the  Lipshitz condition of order 1 (defined by (\ref{lip})), and let the characteristic functions $\Phi_ {X_{i_1},  X_{i_2}, X_{i_3}, X_{i_4}}$ and $\Phi_ {X_{i_1}',  X_{i_2}, X_{i_3}, X_{i_4}}$ be  absolutely integrable. Then, for any $T > 0$, 
\begin{align}\label{theorem31}
& \underset{x_1, x_2, x_3, x_4}{sup}{|f_{X_{i_1},  X_{i_2}, X_{i_3}, X_{i_4}} (x_1, x_2, x_3, x_4) - f_{X_{i_1}',  X_{i_2}, X_{i_3}, X_{i_4}}(x_1, x_2, x_3, x_4)|} \nonumber \\
& \le C\frac{T^6}{(2 \pi)^4} [Cov({X}_{i_1}, {X}_{i_2})+ Cov({X}_{i_1}, {X}_{i_3}) + Cov({X}_{i_1}, {X}_{i_4})] + \frac{16C\sqrt{3}}{T},
\end{align}
where C is an absolute constant.
Solving for an optimal $T>0$, we get,
\begin{align} \label{theorem32}
&  \underset{x_1, x_2, x_3, x_4}{sup}{|f_{X_{i_1},  X_{i_2}, X_{i_3}, X_{i_4}} (x_1, x_2, x_3, x_4) - f_{X'_{i_1},  X_{i_2}, X_{i_3}, X_{i_4}}(x_1, x_2, x_3, x_4)| }\nonumber \\ 
& \le C [Cov({X}_{i_1}, {X}_{i_2})^{1/7}+ Cov({X}_{i_1}, {X}_{i_3})^{1/7} + Cov({X}_{i_1}, {X}_{i_4})^{1/7}],
\end{align}
where C is an absolute constant.
\end{lemma}

\begin{proof}
 Let $\textbf{t} = (t_1, t_2, t_3, t_4) \in \mathbb{R}^4$. Using Lemma $\ref{lemma61}$, we get
\begin{align}\label{charlemma3}
 & |\Phi_{X_{i_1},  X_{i_2}, X_{i_3}, X_{i_4}}(\textbf{t}) - \Phi_{X'_{i_1},  X_{i_2}, X_{i_3}, X_{i_4}}(\textbf{t}) |  \nonumber \\
 & \le C[|t_1t_2| Cov({X}_i, {X}_j) + |t_1t_3| Cov({X}_i, {X}_k) + |t_1t_4|Cov({X}_i, {X}_l)].
\end{align}
Using Lemma  $\ref{lemma62}$ and $(\ref{charlemma3})$, we get $(\ref{theorem31})$. 

Putting $T = [Cov({X}_{i_1}, {X}_{i_2})+ Cov({X}_{i_1}, {X}_{i_3}) + Cov({X}_{i_1}, {X}_{i_4}) ]^{-1/7}$ in $(\ref{theorem31})$, we get $(\ref{theorem32})$. 
\end{proof}

 \textbf{Proof of Theorem $\ref{theorem25}$}. 
 \begin{proof}   Define    for all $\textbf{x}_{i, j, k, l} = (x_i, x_j, x_k, x_l)\in [-C_1, C_1]^4$,
  \begin{equation*} 
  f_{i, (j, k, l)} = f_{X_i,  X_j, X_k, X_l}(\textbf{x}_{i, j, k, l}) - f_{X'_i, X_j,X_k, X_l}(\textbf{x}_{i, j, k, l}).
  \end{equation*}
Using Lemma $\ref{pdflemma2}$, and under the assumption $(T1)$ given in Section $\ref{CLT}$, we get, for all distinct $i, j, k, l$, such that $1 \le i < j \le n$ and $1 \le k < l \le n$,
\begin{align}
& |E( h^{(2)}(X_i, X_j) h^{(2)}(X_k, X_l))| \nonumber \\
& =   |E( h^{(2)}(X_i, X_j) h^{(2)}(X_k, X_l)) - E( h^{(2)}(X_i', X_j) h^{(2)}(X_k, X_l))| \label{546} \\
 & =  | \int_{[-C_1, C_1]^4} f_{i, (j, k, l)}dx_idx_jdx_kdx_l  \le CC_1^4 ||f_{i, (j, k, l)}||_{\infty} \nonumber \\
& \le C(Cov({X}_i, {X}_j)^{1/7} + Cov({X}_i, {X}_k)^{1/7}  + Cov({X}_i, {X}_l)^{1/7} ).  \label{theorem34}
\end{align}
The equality in $(\ref{546})$ follows  as by definition $h^{(2)}(x, y)$ is a degenerate kernel. The  inequality in $(\ref{theorem34})$ follows from $(\ref{theorem32})$. Similarly, 
\begin{align}
& |E( h^{(2)}(X_i, X_j) h^{(2)}(X_k, X_l))|  \le C(Cov({X}_j, {X}_i)^{1/7} + Cov({X}_j, {X}_k)^{1/7}  + Cov({X}_j, {X}_l)^{1/7} )  \text{ and } \label{theorem35}\\
& |E( h^{(2)}(X_i, X_j) h^{(2)}(X_k, X_l))| \le C(Cov({X}_k, {X}_j)^{1/7} + Cov({X}_k, {X}_i)^{1/7}  + Cov({X}_k, {X}_l)^{1/7} ).  \label{theorem36}
\end{align}
Combining $(\ref{theorem34})$, $(\ref{theorem35})$ and $(\ref{theorem36})$,
\begin{equation} \label{theorem3alldiff}
 |E( h^{(2)}(X_i, X_j) h^{(2)}(X_k, X_l))| \le CT^{1/3}.
\end{equation}
where, $T = [Cov({X}_i, {X}_j)^{1/7} + Cov({X}_i, {X}_k)^{1/7}  + Cov({X}_i, {X}_l)^{1/7}]\times[Cov({X}_j, {X}_i)^{1/7} +\\ Cov({X}_j, {X}_k)^{1/7}  + Cov({X}_j, {X}_l)^{1/7}]\times[Cov({X}_k, {X}_j)^{1/7} + Cov({X}_k, {X}_i)^{1/7}  + Cov({X}_k, {X}_l)^{1/7}]$.

Next, assume that there are 3 distinct indices in  $i, j, k, l$, such that, $1 \le i < j \le n$ and $1 \le k < l \le n$. For example, assume $j = k$, then using $(T2)$,
\begin{align} \label{3assfchap5}
|E( h^{(2)}(X_i, X_j) h^{(2)}(X_j, X_l))|  \le C(Cov({X}_i, {X}_j)^{1/7} + Cov({X}_i, {X}_l)^{1/7} ).
\end{align}
Similarly, we can calculate for other combinations with 3 distinct  indices in $i, j, k, l$.

Note that, as $h^{(2)}(x, y)$ is bounded, 
\begin{align} \label{theorem3same}
\sum_{1 \le i < j \le n}  |E( h^{(2)}(X_i, X_j)^2)|  = O(n^2).
\end{align}

Hence, from $(\ref{theorem3alldiff})$,  $(\ref{3assfchap5})$ and $(\ref{theorem3same})$, and using $\sum_{j=1}^\infty Cov({X}_1, {X}_j)^{1/21}$ $<$ $\infty$, we get 
\begin{equation}\label{theorem3same11}
\sum_{1 \le i < j \le n} \sum_{1 \le k < l \le n} |E( h^{(2)}(X_i, X_j)h^{(2)}(X_k, X_l))|  = O(n^2).
\end{equation}
 Using the Hoeffding's decomposition for $U_n(\rho)$ and the  central limit theorem for stationary functions of associated random variables given in Theorem $17$ of \cite{MR789244}, rest of the proof follows similarly as the proof of Theorem 3.6 of \cite{GargdewanSPL2}. 
\end{proof}

\textbf{Proof of Theorem $\ref{theorem41}$}
\begin{itemize}
\item[(i)] For Deshpande's test  we took $\hat{\rho}_1(x) = \frac{F_n(x/b) + 1 - F_n(xb)}{2}$. Putting $\psi_n  = O(n^{1/4})$ and $s > 6$ in Lemma $\ref{roussaspaper}$, we get,  $|B_n - \hat{B}_n|$ $\to 0$ a.s as $n \to \infty$.
\item[(ii)] For Hollander and Proschan's test, we took $\hat{\rho}_1(x) =  \frac{1}{3}\Big( \sum_{i:X_i \le x} F_n(x-X_i)/n + \sum_{i=1}^n\bar{F}_n(x+X_i)/n + \\ \sum_{i: X_i \ge x} F_n(X_i-x)/n \Big)$.  Observe that,
\begin{align}
& \underset{x \in [0, C_1]} {sup} \Big| {\sum_{i =1}^n \frac{(F_n(x - X_i) - E(F(x - X_i)))}{n^{3/4}}} \Big|  \le C\underset{y \in [0, C_1]} {sup} \Big| {\sum_{i =1}^n \frac{F_n(y) - F(y)}{n^{3/4}}} \Big|  \nonumber \\  
& \le C\underset{y \in [0, C_1]} {sup} n^{1/4}\Big| F_n(y) - F(y) \Big|   \to 0 \text{ a.s. as $n \to \infty$,}
\end{align}
 under the assumption of $s > 6$ and putting $\psi_n =  O(n^{1/4})$ in Lemma $\ref{roussaspaper}$. The convergence of the last two terms follow similarly. 
\item[(iii)] For Ahmad's test, we took $\hat{\rho}_1(x) = 2x\bar{F}_n(x) - \frac{x}{2} + \frac{3\bar{X}_n}{2} - 2 \frac{\sum_{i =1}^n X_i I(X_i > x)}{n}$. The convergence of the first 2 terms follows easily. For the last 2 terms, observe that
 \begin{equation}
 \Big| \frac{\sum_{i =1}^n (X_i - E(X_i))}{n^{3/4}} \Big| \le C_1 C\underset{y \in [0, C_1]} {sup} n^{1/4}\Big| F_n(y) - F(y) \Big| \to 0 \text{ a.s. as $n \to \infty$}
\end{equation}
and
 \begin{align}
& \underset{x \in [0, C_1]} {sup}\Big| \frac{\sum_{i =1}^n (X_iI(X_i > x) - E(X_iI(X_i > x))}{n^{3/4}} \Big| \le C_1 C\underset{y \in [0, C_1]} {sup} n^{1/4}\Big| F_n(y) - F(y) \Big| \to 0 \text{ a.s. as $n \to \infty$.} 
\end{align}
using Lemma $\ref{roussaspaper}$ with $s > 6$ and $\psi_n =  O(n^{1/4})$.

\end{itemize}

\section*{Acknowledgements}

This work was done when the first author was working as a research fellow at Indian Statistical Institute, Delhi. The author would like to  thank the institute for all the facilities.

\bibliographystyle{yearfirst}
\setlength{\bibsep}{0.0pt}
{\footnotesize\bibliography{References}}

\end{document}